\documentclass[lettersize,journal,twoside]{style/IEEEtran}

\IEEEoverridecommandlockouts

\usepackage[hidelinks,bookmarksopen,bookmarksnumbered]{hyperref}

\usepackage[utf8]{inputenc}
\usepackage{xcolor}
\usepackage{booktabs}
\usepackage{multirow}
\usepackage{graphicx}
\usepackage{times}
\usepackage{amsmath, amsthm}
\usepackage{amssymb}
\usepackage{bm}
\usepackage{microtype}
\usepackage{mathtools}
\usepackage[ruled,algo2e]{algorithm2e}
\usepackage{cite}
\usepackage{adjustbox}
\usepackage{enumitem}
\usepackage{xparse}
\usepackage{pifont}

\newcommand{\mcal}[1]{\mathcal{#1}}
\newcommand{\T}{\mathsf{T}}

\newcommand{\norm}[1]{\left\Vert #1 \right\Vert}
\newcommand{\mbb}[1]{\mathbb{#1}}

\newcommand{\diag}{\mathrm{diag}}

\newcommand{\minimize}[1]{ \ensuremath{\underset{#1}{\mathrm{minimize}}\ }}
\newcommand{\algsep}{\\[0.4em]}
\newcommand*{\mean}[1]{\ensuremath{\overline{#1}}}
\newcommand*{\trace}[1]{\ensuremath{\mathrm{trace}\left({#1}\right)}}

\newcommand*{\subj}{\ensuremath{\mathrm{subject\ to\ }}}
\newcommand*{\bp}[1]{\ensuremath{\Bigl(#1\Bigr)}}

\newtheorem{assumption}{Assumption}

\newtheorem{theorem}{Theorem}
\newtheorem{corollary}{Corollary}

\newtheorem{lemma}{Lemma}
\newtheorem{remark}{Remark}

\NewDocumentCommand{\card}{sO{}m}{%
  {\IfBooleanTF{#1}
    {\oldnormaux{\left|\right.}{\left.\right|}{#3}}
    {\oldnormaux{#2|}{#2|}{#3}}}
}
\makeatletter
\newcommand{\oldnormaux}[3]{\mathpalette\oldnormaux@i{{#1}{#2}{#3}}}
\newcommand{\oldnormaux@i}[2]{\oldnormaux@ii#1#2}
\newcommand{\oldnormaux@ii}[4]{%
  \sbox\z@{$\m@th#1#2#4#3$}%
  \sbox\tw@{$\m@th\|$}%
  \mathopen{\hbox to\wd\tw@{\hss\vrule height \ht\z@ depth \dp\z@ width .2\wd\tw@\hss}}%
  #4
  \mathclose{\hbox to\wd\tw@{\hss\vrule height \ht\z@ depth \dp\z@ width .2\wd\tw@\hss}}%
}
\makeatother

\makeatletter
\newcommand{\longdash}[1][2em]{%
  \makebox[#1]{$\m@th\smash-\mkern-7mu\cleaders\hbox{$\mkern-2mu\smash-\mkern-2mu$}\hfill\mkern-7mu\smash-$}}
\makeatother
\newcommand{\omitskip}{\kern-\arraycolsep}
\newcommand{\llongdash}[1][2em]{\longdash[#1]\omitskip}
\newcommand{\rlongdash}[1][2em]{\omitskip\longdash[#1]}
\newcommand{\cmark}{\ding{51}}
\newcommand{\xmark}{\ding{55}}

\title{\LARGE \bf Distributed Quasi-Newton Method for Multi-Agent Optimization}

\author{Ola Shorinwa$^{1}$ and Mac Schwager$^{2}$%
	\thanks{*This work was supported in part by ONR grants N00014-23-1-2354 and N00014-18-1-2830, and DARPA grant  HR001120C0107. Toyota Research Institute provided funds to support this work.}%
	\thanks{$^{1}$Ola Shorinwa is with the Department of Mechanical Engineering, Stanford University, CA, USA {\tt\small shorinwa@stanford.edu}.}%
	\thanks{$^{2}$Mac Schwager is with the Department of Aeronautics and Astronautics Engineering, Stanford University, CA, USA {\tt\small schwager@stanford.edu}.}%
}

\begin{document}
	\maketitle

	\begin{abstract}
		We present a \emph{distributed quasi-Newton} (DQN) method, which enables a group of agents to compute an optimal solution of a \emph{separable multi-agent} optimization problem locally using an approximation of the curvature of the aggregate objective function. Each agent computes a descent direction from its local estimate of the aggregate Hessian, obtained from quasi-Newton approximation schemes using the gradient of its local objective function. Moreover, we introduce a distributed quasi-Newton method for \emph{equality-constrained} optimization (\mbox{EC-DQN}), where each agent takes \emph{Karush-Kuhn-Tucker}-like update steps to compute an optimal solution. In our algorithms, each agent communicates with its one-hop neighbors over a \emph{peer-to-peer} communication network to compute a common solution. We prove convergence of our algorithms to a stationary point of the optimization problem. In addition, we demonstrate the competitive empirical convergence of our algorithm in both \emph{well-conditioned} and \emph{ill-conditioned} optimization problems, compared to existing distributed \emph{first-order} and \emph{second-order} methods. Particularly, in \emph{ill-conditioned} problems, our algorithms achieve a faster computation time for convergence, while requiring a lower communication cost, across a range of communication networks with different degrees of connectedness.
	\end{abstract}


	\section{Introduction}
\label{sec:introduction}
The paradigm of optimization has been widely employed in model design and analysis \cite{bashir2021aerodynamic, masdari2019optimization, harish2016reduced, ahmadianfar2021gradient, long2021parameters}, planning \cite{toumieh2022decentralized, torreno2014fmap, mishra2019multi}, and control \cite{franze2018distributed, wang2014synthesis, luo2017multi}, across broad problem domains. In many situations, the pertinent problem requires optimization over spatially-distributed data and agent-specific objective and constraints across a group of agents---referred to as the \emph{multi-agent} problem. Privacy constraints and resource limitations often make aggregation of the problem data for centralized optimization intractable, thus rendering distributed optimization as a viable alternative for optimization. Via distributed optimization, each agent communicates with its local neighbors over a communication network to compute a common, optimal solution of the problem, without any access to the aggregate problem data over all agents. In this work, we not only consider \emph{unconstrained} multi-agent problems, we also consider \emph{equality-constrained} problems, noting that many existing distributed optimization algorithms (e.g., the broad class of distributed first-order methods) typically consider only unconstrained multi-agent optimization problems. Meanwhile, constrained multi-agent optimization problems arise in many settings, e.g., basis pursuit problems in signal processing.

In this work, we focus on quasi-Newton methods, which are particularly useful in \emph{ill-conditioned} optimization problems, where first-order methods perform poorly. Moreover, quasi-Newton methods generally provide faster convergence rates compared to first-order methods by leveraging information on the curvature of the objective function, without the additional overhead of Newton's method.
In contrast to the existing distributed quasi-Newton methods D-BFGS \cite{eisen2017decentralized} and PD-QN \cite{eisen2019primal}, which utilize an approximation of the problem Hessian over a ``local" neighborhood, we introduce distributed quasi-Newton methods---for unconstrained (\emph{DQN}) and equality-constrained optimization (\emph{EC-DQN})---where each agent utilizes an estimate of the \emph{aggregate} Hessian, which considers the objective functions of \emph{all} agents. Consequently, our algorithms can be considered to be analogous to distributed gradient-tracking methods in the quasi-Newton optimization domain. We discuss our method in the context of existing distributed quasi-Newton methods in greater detail in Section~\ref{sec:related_work}. By directly estimating the inverse Hessian of the aggregate optimization problem, DQN circumvents the computational cost associated with matrix inversion procedures arising in Newton (or quasi-Newton) methods. We prove convergence of our algorithms to a stationary solution of the aggregate optimization problem, provided the objective function is convex and smooth.

We examine the performance of DQN and EC-DQN in unconstrained and equality-constrained optimization problems, in terms of the computation time required for convergence and the associated communication cost, compared to existing methods. We demonstrate that our algorithms perform competitively in well-conditioned problems, and in particular, in ill-conditioned problems, converge faster than other algorithms across different communication networks while providing about a seven-to-ten-times reduction in the communication cost in the case of DQN.
We summarize our contributions as follows:
\begin{enumerate}
    \item We derive a \emph{distributed quasi-Newton} (DQN) method for unconstrained multi-agent optimization, enabling each agent to utilize an estimate of the aggregate problem Hessian for faster convergence, especially in \emph{ill-conditioned} problem settings.
    
    \item Further, we present a distributed quasi-Newton method for \emph{equality-constrained} optimization problems \emph{EC-DQN}, which extends DQN to the constrained setting, without any associated projection procedure for constraint satisfaction.

    \item We demonstrate the superior performance of our algorithms DQN and EC-DQN in a variety of problems across a range of communication networks, particularly in ill-conditioned problems. 
\end{enumerate}

The paper is organized as follows: In Section~\ref{sec:related_work}, we present the related work and, in Section~\ref{sec:preliminaries}, present notation and introduce relevant concepts fundamental to understanding the discussion in this paper. In Section~\ref{sec:problem}, we formulate the distributed optimization problem. We derive our distributed optimization algorithms in Sections~\ref{sec:distributed_alg_uncons}~and~\ref{sec:distributed_alg_cons}. We examine the performance of our algorithms in simulation in Sections~\ref{sec:simulations_uncons}~and~\ref{sec:simulations_cons}. In Section~\ref{sec:conclusion}, we provide concluding remarks.

	 \section{Related Work}
 \label{sec:related_work}

 In this section, we review notable algorithms for finite-sum distributed optimization problems and refer readers to
 \cite{yang2019survey, shorinwa2023distributedsurvey} for a more detailed overview of these methods.
 The simplest distributed optimization algorithms extend the well-known \emph{centralized gradient descent} algorithm to the distributed setting \cite{nedic2009, lobel2010distributed}, using each agent's local gradient. However, these methods do not converge to an optimal solution of the aggregate optimization problem under a fixed step-size scheme. Distributed gradient-tracking methods \cite{shi2015extra, nedic2017achieving, liao2022compressed, li2021accelerated, lu2020nesterov, sun2022distributed} attempt to address this fundamental limitation by utilizing an estimate of the gradient of the aggregate objective function. In the best-case scenario, i.e., in problems with strongly-convex objective functions with Lipschitz-continuous gradients, distributed first-order methods achieve linear convergence to an optimal solution of the problem. However, in general, these algorithms only apply to \emph{separable} \emph{unconstrained} optimization problems.

 Distributed \emph{primal-dual} methods seek to find a saddle-point of the Lagrangian, provided one exists, by alternating between a minimization procedure to compute a primal iterate and a maximization procedure to compute a dual iterate, where the minimization problems are solved via gradient-descent \cite{mateos2016distributed, cortes2019distributed, aybat2016primal}, or explicitly via the \emph{alternating direction method of multipliers} (ADMM) \cite{mateos2010distributed, mota2013d, chang2014multi, shorinwa2020scalable, carli2019distributed} or other augmented-Lagrangian-based distributed optimization algorithms \cite{zhang2018consensus, jakovetic2014linear, kia2017augmented}. Primal-dual methods solve constrained optimization problems directly, in general; however, like distributed first-order methods, these methods converge linearly to an optimal solution of the problem, in the best-case setting.

 Distributed \emph{second-order} methods extend Newton's method to distributed optimization, utilizing higher-order information, such as the Hessian of the objective function \cite{mokhtari2016network, liu2023communication, mansoori2019fast}. As a result, these methods generally provide faster convergence rates (superlinear to quadratic convergence rates) compared to other distributed optimization methods. However, the improved rate of convergence often comes at the cost of a greater computational overhead, and in some situations, computation of the problem Hessian is outrightly infeasible. Distributed \emph{quasi-Newton} methods attempt to resolve this challenge by rather utilizing an estimate of the problem Hessian, rendering computation of the exact Hessian unnecessary \cite{eisen2017decentralized, soori2020dave, mokhtari2016decentralized, eisen2019primal, li2022communication}. 

 The distributed quasi-Newton methods \cite{eisen2017decentralized, eisen2019primal} are of particular relevance to this work. In the D-BFGS method \cite{eisen2017decentralized}, each agent estimates the problem Hessian over a ``local" neighborhood, which does not directly account for the objectives of all agents. In contrast, in our distributed algorithm, each agent computes an approximation of the ``aggregate" problem Hessian, i.e., across all agents (regardless of their connectedness). Prior work on first-order methods has shown that utilizing an estimate of the aggregate gradient instead of the local gradient leads to faster convergence, which motivated the development of our method. We explore this potential competitive advantage in our work. In addition, the ``local" Hessian ${B_{i} \in \mbb{R}^{m_{i}n}}$ (where $m_{i}$ represents the number of neighbors of agent $i$ and $n$ represents the dimension of the optimization variable) scales poorly with the degree of connectedness of the graph. In the worst-case, each agent computes the inverse of a symmetric, square matrix of dimension ${Nn}$-by-${Nn}$ (where $N$ is the number of agents), which has a complexity of ${O(N^{3}n^{3})}$. In contrast, in our method, even in the worst-case implementation, the estimated Hessian is always of dimension $n$-by-$n$, giving a complexity of ${O(n^{3})}$, a factor of $N^{3}$ speedup. Moreover, in the unconstrained case, our method does not require direct computation of the inverse of the Hessian; rather, each agent directly estimates the inverse Hessian, eliminating the ${O(n^{3})}$ computational complexity. Further, the D-BFGS method requires agent $i$ to communicate an aggregate variable of dimension ${(m_{i} + 2)n}$ per iteration, which, again, scales poorly in densely-connected networks compared to our method, which requires the communication of a variable of dimension ${3n}$ per iteration. 
 
 The PD-QN method in \cite{eisen2019primal} extends \cite{eisen2017decentralized} to the primal-dual setting, using a formulation based on the augmented Lagrangian. In the primal procedure in \cite{eisen2019primal}, each agent utilizes a Taylor Series expansion to approximate the inverse Hessian of the augmented Lagrangian, which is constructed from the local Hessian of each agent. This procedure requires a nested-loop for the $K$ terms in the Taylor Series expansion, contributing additional computation and communication overhead. In contrast, our method takes a different approach by directly estimating the ``aggregate" problem Hessian and does not require any nested-loops. Moreover, the dual procedure in PD-QN\cite{eisen2019primal} utilizes the D-BFGS method \cite{eisen2017decentralized} and thus suffers from the same aforementioned challenges. Lastly, all these methods apply only to unconstrained distributed optimization. In this work, we derive a distributed algorithm for the equality-constrained optimization setting.

		\section{Notation and Preliminaries}
\label{sec:preliminaries}
We present relevant notation used in this paper. Given an objective function ${f: \mbb{R}^{n} \rightarrow \mbb{R}}$, we denote its gradient by ${\nabla f: \mbb{R}^{n} \rightarrow \mbb{R}^{n}}$ and its Hessian by ${\nabla^{2} f: \mbb{R}^{n} \rightarrow \mbb{R}^{n \times n}}$. Further, we denote the domain of a function $f$ by ${\mathrm{dom}(f)}$. We denote the spectral norm of a matrix ${A \in \mbb{R}^{m \times n}}$ by $\rho(A)$, representing its largest singular value. We denote the $p$-norm of a matrix by $\norm{\cdot}_{p}$, where ${p = F}$ denotes the Frobenius norm. We denote the $n$-dimensional all-ones vector by $\bm{1}_{n}$ and the Identity matrix by ${I_{n} \in \mbb{R}^{n \times n}}$. We denote the non-negative orthant by $\mbb{R}_{+}$, the strictly-positive orthant by $\mbb{R}_{++}$, and the set of symmetric positive-definite matrices by $\mbb{S}_{++}^{n}$. We define the (row-wise) mean of a matrix ${A \in \mbb{R}^{m \times n}}$ as ${\mean{A} = \frac{1}{m}\bm{1}_{m}\bm{1}_{m}^{\T} A}$, where ${\mean{A} \in \mbb{R}^{m \times n}}$.

In this paper, we represent the agents as nodes in a static, undirected communication graph ${\mcal{G} = (\mcal{V}, \mcal{E})}$, defined by a set of vertices ${\mcal{V} = \{1, \ldots, N\}}$, denoting the agents, and a set of edges ${\mcal{E} \subset \mcal{V} \times \mcal{V}}$, signifying the existence of a communication link between a pair of agents. For example, edge ${e = (i, j)}$ indicates that agent $i$ can communicate with agent $j$. Further, we denote the set of neighbors of agent $i$, including agent $i$, as ${\mcal{N}_{i} = \{i \} \cup \{(i, j),\ \forall j \in \mcal{V} \mid (i, j) \in \mcal{E} \}}$. In addition, we associate a \emph{mixing} matrix ${W \in \mbb{R}^{N \times N}}$ with the communication graph $\mcal{G}$. A mixing matrix is \emph{compatible} with $\mcal{G}$, if ${w_{ij} = 0,\ \forall j \notin \mcal{N}_{i},\ \forall i \in \mcal{V}}$.

	\section{Problem Formulation and Centralized Quasi-Newton Methods}
\label{sec:problem}
We consider a problem setting with $N$ agents, where each agent has access to its local problem data (which may have been collected from its onboard sensors). In the unconstrained multi-agent optimization problem setting, the group of agents seeks to \emph{collaboratively} compute a solution to the \emph{aggregate} optimization problem, defined over the entire data across all agents, given by:
\begin{equation}
	\label{eq:global_prob_uncons}
	\minimize{x \in \mathbb{R}^{n}} \frac{1}{N} \sum_{i = 1}^{N} f_{i}(x),
\end{equation}
where the objective function comprises of a sum of the individual objective functions of all agents, defined over the optimization variable denoted by ${x \in \mbb{R}}$. In \eqref{eq:global_prob_uncons}, ${f_{i} : \mbb{R}^{n} \mapsto \mbb{R}}$ denotes the local objective function of agent $i$, which depends only on its local data, such as its local observations, measurements, and preferences, and does not involve the local problem data of other agents.

In the constrained optimization setting, each agent seeks to compute an optimal solution of the \emph{convex} equality-constrained optimization problem, given by:
\begin{equation}
	\label{eq:global_prob_cons}
	\begin{aligned}
		\minimize{x \in \mbb{R}^{n}} &\sum_{i = 1}^{N} f_{i}(x) \\
		\subj &Ax = b,
	\end{aligned}
\end{equation}
where ${f_{i}: \mbb{R}^{n} \mapsto \mbb{R}}$, denotes the local convex objective function of agent $i$. The solution of \eqref{eq:global_prob_cons} must satisfy the equality constraint in \eqref{eq:global_prob_cons}, with ${A \in \mbb{R}^{m \times n}}$ and ${b \in \mbb{R}^{m}}$, which could encode domain-specific prior knowledge or requirements. We assume that each agent has access to $A$ and $b$. Such situations arise in practice, e.g., when the constraints defined by $A$ and $b$ results from functional or safety requirements that must be satisfied by the solution. Such constraints arise in a variety of disciplines, e.g., in controls and signal processing domains, where the solution (signal) should be compatible with the observed measurements, given the measurement model, in the noiseless case.

Overloading notation, we denote the composite objective function in \eqref{eq:global_prob_uncons} and \eqref{eq:global_prob_cons} as ${f: \mbb{R}^{n} \mapsto \mbb{R}}$. In this work, we assume that the local problem data of each agent is inaccessible to other agents, due to privacy or communication limitations. We note that each agent cannot solve \eqref{eq:global_prob_uncons} or \eqref{eq:global_prob_cons} independently, i.e., in isolation, since each agent lacks global access to the entire problem data. Consequently, all the agents must work together to compute an optimal solution, by leveraging distributed optimization techniques.

We state some assumptions on the objective and constraint functions in \eqref{eq:global_prob_uncons} and \eqref{eq:global_prob_cons}, in addition to assumptions on the communication graph representing the communication network among the agents.

\begin{assumption}
    \label{assm:smooth_objective}
    The local objective function of each agent is convex, closed, and proper. Further, $f_{i}$ is coercive and $L_{i}$-smooth, where ${L_{i} \in \mbb{R}_{++}}$,~${\forall i \in \mcal{V}}$.
\end{assumption}

This assumption implies that $\nabla f_{i}$ is Lipschitz-continuous and further implies that $f$ is $L$-smooth, where ${L \geq \max_{i \in \mcal{V}} L_{i}}$.

\begin{assumption}
    \label{assm:solvable_problem}
    The multi-agent optimization problems \eqref{eq:global_prob_uncons} and \eqref{eq:global_prob_cons} have a non-empty feasible set, with an optimal solution $x^{\star}$.
\end{assumption}

This assumption eliminates scenarios with unbounded and infeasible optimization problems.

\begin{assumption}
    \label{assm:mixing_matrix}
    The mixing matrix $W$ is doubly-stochastic with ${W \bm{1}_{N} = \bm{1}_{N}}$ and ${\bm{1}_{N}^{\T} W = \bm{1}_{N}}$. In addition, the matrix ${M = W - \frac{\bm{1}_{N} \bm{1}_{N}^{\T}}{N}}$ has its maximum singular value less than one.
\end{assumption}

Various schemes exist for generating mixing matrices that satisfy this assumption, such as schemes that utilize Metropolis-Hasting weights, given by:
\begin{equation*}
    w_{ij} = \begin{cases*}
                    \frac{1}{\max\{\deg(i), \deg(j)\} + \epsilon}, & if $(i,j) \in \mcal{E},$ \\
                    0 & if $(i,j) \notin \mcal{E}$ and $i \neq j,$ \\
                    1 - \sum_{r \in \mcal{V}} w_{ir} & if $i = j,$
               \end{cases*}
\end{equation*}
where ${\epsilon > 0}$ represents a sufficiently small positive constant.
From Assumption~\ref{assm:mixing_matrix}, given that the spectral norm of the matrix ${M = W - \frac{\bm{1}_{N} \bm{1}_{N}^{\T}}{N}}$ is less than one, we note that:
\begin{equation}
    \lim_{k \rightarrow \infty} W^{k} = \frac{\bm{1}_{N} \bm{1}_{N}^{\T}}{N},
\end{equation}
i.e., consensus is attained when the mixing matrix $W$ is utilized in a linear-consensus scheme.

Before discussing centralized quasi-Newton methods, we provide a brief overview of Newton's method for optimization problems, without a line-search procedure. Newton's method seeks to compute the extremum of the optimization problem in \eqref{eq:global_prob_uncons} by generating a sequence of iterates through the recurrence:
\begin{equation}
	\label{eq:newton_recurrence}
	x^{(k + 1)} = x^{(k)} - \left( \nabla^{2} f(x^{(k)})\right)^{-1}  \nabla f(x^{(k)}),
\end{equation}
which requires the evaluation of the gradient and Hessian of $f$, in addition to computing the inverse of the Hessian. Generally, Newton's method yields high-accuracy solutions within only a few iterations, when it converges, making it a highly effective optimization method. However, in large-scale problems (e.g., with more than $10^4$ decision variables) the computation of the inverse Hessian is particularly challenging, limiting the application of Newton's method to relatively small problems.

Quasi-Newton methods were developed to address these limitations. These methods approximate the Hessian (or its inverse), eliminating the need for its explicit, exact computation. Examples of quasi-Newton methods include the Broyden-Fletcher-Goldfarb-Shanno (BFGS) method \cite{dennis1977quasi}, Davidon-Fletcher-Powell (DFP) method \cite{davidon1959variable}, and the Symmetric Rank-One (SR1) method \cite{byrd1996analysis}. Here, we describe the BFGS and note similarities between the BFGS method and other quasi-Newton methods. Quasi-Newton methods compute an approximate Hessian $B^{(k + 1)}$ that satisfies the \emph{secant equation}, given by:
\begin{equation}
	B^{(k + 1)}s^{(k)} = y^{k},
\end{equation}
at iteration $k$, where ${s^{(k)} = x^{(k + 1)} - x^{(k)}}$ and ${y^{(k)} = \nabla f (x^{(k + 1)}) - \nabla f (x^{(k)})}$. In addition, noting the symmetry of the Hessian, these methods seek a symmetric approximation of the Hessian, that is also positive-definite, to guarantee that a solution exists for the recurrence in \eqref{eq:newton_recurrence}, when the approximate Hessian is used in place of the exact Hessian. We note that the approximate Hessian computed by the SR1 method is not guaranteed to be positive-definite. The BFGS method computes the approximate inverse Hessian $C^{(k + 1)}$ as the solution of the optimization problem:
\begin{equation}
	\label{eq:BFGS_problem}
	\begin{aligned}
		\minimize{C} &\norm{\Lambda^{-1}(C - C^{(k)})\Lambda^{-\T}}_{F} \\
		\subj &C = C^{\T} \\
		&Cy^{(k)} = s^{(k)},
	\end{aligned}
\end{equation}
where ${C \in \mbb{R}^{n \times n}}$, and ${\Lambda \in \mbb{R}^{n \times n}}$ denotes a nonsingular matrix, satisfying the equality ${\Lambda\Lambda^{\T}y = s}$. This equality ensures that ${y^{\T}s > 0}$, referred to as the \emph{curvature condition}. Satisfaction of the curvature condition guarantees positive-definiteness of the resulting approximate (inverse) Hessian. The BFGS update scheme for the approximate inverse Hessian simplifies to:
\begin{equation}
	\label{eq:BFGS_update}
	\begin{aligned}
		C^{(k + 1)} &= \left(I_{n} - \frac{s^{(k)}y^{(k) \T}}{y^{(k) \T} s^{(k)}}  \right) C^{(k)} \left(I_{n} - \frac{y^{(k)}s^{(k) \T}}{y^{(k) \T} s^{(k)}}  \right) \\
		& \quad + \frac{s^{(k)}s^{(k) \T}}{y^{(k) \T} s^{(k)}}.
	\end{aligned}
\end{equation}
Similarly, the DFP update scheme can be derived by solving a related optimization problem, with the resulting update scheme given by:
\begin{equation}
	\label{eq:DFP_update}
	\begin{aligned}
		C^{(k + 1)} &= C^{(k)} - \frac{C^{(k)}y^{(k)}y^{(k)\T}C^{(k)}}{y^{(k)\T}C^{(k)}y^{(k)}} + \frac{s^{(k)}s^{(k)\T}}{y^{(k)\T}s^{(k)}}.
	\end{aligned}
\end{equation}
Quasi-Newton methods compute the minimizer of the optimization problem in \eqref{eq:global_prob_uncons} via the following recurrence:
\begin{equation}
	\label{eq:quasi_newton_recurrence}
	x^{(k + 1)} = x^{(k)} - \alpha^{(k)} C^{(k)}  \nabla f(x^{(k)}),
\end{equation}
where ${\alpha^{k} \in \mbb{R}_{++}}$ represents the step-size at iteration $k$.

	\section{Distributed Unconstrained Optimization}
\label{sec:distributed_alg_uncons}
To solve the optimization problem in \eqref{eq:global_prob_uncons}, quasi-Newton methods require the computation of the approximate Hessian $C^{(k)}$ to execute the recurrence in \eqref{eq:quasi_newton_recurrence}, which involves the local data of all agents through the gradient of $f$. However, no agent has access to all the problem data. Aggregation of the problem data at a central node and subsequent optimization pose notable challenges, as discussed in Section~\ref{sec:introduction}. Consequently, in this section, we derive a distributed quasi-Newton method for unconstrained optimization in multi-agent networks, which eliminates the need for a central station, while enabling each agent to compute an optimal solution of \eqref{eq:global_prob_uncons}, via local computation and communication with its immediate neighbors.

We begin by assigning local variables ${x_{i} \in \mbb{R}^{n}}$ to each agent, representing a copy of the shared optimization variable ${x}$. Further, we introduce the following \emph{aggregate objective function} and \emph{aggregate problem variable}:
\begin{align*}
    \bm{f}(\bm{x}) &= \sum_{i = 1}^{N} f_{i}(x_{i}), \quad
    \bm{x} = \begin{bmatrix}
        \llongdash & x_{1}^{\T} & \rlongdash \\
        & \vdots &  \\
        \llongdash & x_{N}^{\T} & \rlongdash
        \end{bmatrix}, %
\end{align*}
where each component of ${f: \mbb{R}^{N \times n} \mapsto \mbb{R}}$ depends on the local optimization variable of each agent, with the concatenation of the local problem variables denoted by ${\bm{x} \in \mbb{R}^{N \times n}}$. In addition, we denote the concatenation of the local gradient of all agents by
\begin{align*}
    \bm{g}(\bm{x}) &= \begin{bmatrix}
    \llongdash & \left(\nabla f_{1}(x_{1})\right)^{\T} & \rlongdash \\
    & \vdots & \\
    \llongdash & \left(\nabla f_{N}(x_{N})\right)^{\T} & \rlongdash
    \end{bmatrix}.
\end{align*}
We use the terms: $\bm{g}(\bm{x})$ and $\nabla \bm{f}(\bm{x})$ interchangeably.
To compute an estimate of the inverse Hessian in  \eqref{eq:BFGS_update} (similarly, \eqref{eq:DFP_update}), each agent needs to compute the difference in successive evaluations of the gradient of $f$, denoted by $y$, and the difference in successive iterates of the problem variable $x$, denoted by $s$, which cannot be computed, given the problem setup. As a result, we leverage dynamic average consensus \cite{zhu2010discrete} to enable each agent to compute a local estimate of the gradient of the objective function $f$ in \eqref{eq:global_prob_uncons}, denoted by ${v_{i} \in \mbb{R}^{n}}$. Each agent updates its local estimate of the gradient $v_{i}$ using
\begin{equation}
    v_{i}^{(k + 1)} = \sum_{j \in \mcal{N}_{i}} w_{ij} \left( v_{j}^{(k)} + \nabla f_{j}\bp{x_{j}^{(k + 1)}} - \nabla f_{j}\bp{x_{j}^{(k)}} \right),
\end{equation}
at iteration $k$, using its local estimate of the solution of the optimization problem $x_{i}$. We note that the sequence $\{v_{i}^{(k)}\}_{k \geq 0}$ converges to the limit point
\begin{equation}
    \label{eq:dqn_v_update}
    v_{i}^{(\infty)} = \frac{1}{N} \left(\bm{g}\bp{\bm{x}^{(\infty)}}\right)^{\T} \bm{1}_{N} = \nabla f\bp{x^{(\infty)}},
\end{equation}
provided that the underlying communication graph $\mcal{G}$ is connected and ${\bm{x}^{(k)} \mapsto \bm{x}^{(\infty)}}$ as $k$ approaches infinity.

Now, we can define a fully-distributed update scheme for the approximate Hessian computed by agent $i$, given the local difference variables ${y_{i}^{(k)} = v_{i}^{(k + 1)} - v_{i}^{(k)}}$ and ${s_{i}^{(k)} = x_{i}^{(k + 1)} - x_{i}^{(k)}}$. With these local variables, we can extend the centralized quasi-Newton update schemes to the fully-distributed setting. For example, a fully-distributed variant of the BFGS update scheme in \eqref{eq:BFGS_update} is given by
\begin{equation}
	\label{eq:dis_BFGS_update}
    \begin{aligned}
        C_{i}^{(k + 1)} &= \left(I_{n} - \frac{s_{i}^{(k)}y_{i}^{(k) \T}}{y_{i}^{(k) \T} s_{i}^{(k)}}  \right) C_{i}^{(k)} \left(I_{n} - \frac{y_{i}^{(k)}s_{i}^{(k) \T}}{y_{i}^{(k) \T} s_{i}^{(k)}}  \right) \\
        & \quad + \frac{s_{i}^{(k)}s_{i}^{(k) \T}}{y_{i}^{(k) \T} s_{i}^{(k)}},
    \end{aligned}
\end{equation}
while the analogous DFP update is given by
\begin{equation}
    \label{eq:dis_DFP_update}
    \begin{aligned}
    C_{i}^{(k + 1)} &= C_{i}^{(k)} - \frac{C_{i}^{(k)}y_{i}^{(k)}y_{i}^{(k)\T}C_{i}^{(k)}}{y_{i}^{(k)\T}C_{i}^{(k)}y_{i}^{(k)}} + \frac{s_{i}^{(k)}s_{i}^{(k)\T}}{y_{i}^{(k)\T}s_{i}^{(k)}}.
    \end{aligned}
\end{equation}
In this work, we require  positive-definiteness of $C_{i}^{(k)}$,~${\forall k,\ \forall i \in \mcal{V}}$. This property holds in the BFGS and DFP update schemes if $y_{i}^{(k)\T}s_{i}^{(k)} > 0$. In other cases, to maintain this property, each agent may need to project its estimate of the inverse Hessian to the set of positive-definite matrices $\mbb{S}_{++}^{n}$. In particular, we assume that the eigenvalues of the estimated Hessian is lower-bounded by ${\frac{1}{\gamma}}$, where ${\gamma \in \mbb{R}_{+}}$.

To compute an optimal solution for \eqref{eq:global_prob_uncons}, all agents have to \emph{collaborate} on minimizing their local objective functions, while simultaneously cooperating to achieve consensus, i.e., agreement on their local estimates of the optimal solution. To achieve these goals, we introduce the \emph{distributed quasi-Newton (DQN) algorithm}, given by the recurrence:
\begin{align}
    \bm{x}^{(k + 1)} &= W\left(\bm{x}^{(k)} + \bm{\alpha}^{(k)} \bm{z}^{(k)}\right),  \label{eq:dqn_joint_x_update} \\
    \bm{v}^{(k + 1)} &= W \left(\bm{v}^{(k)} + \nabla \bm{f}\bp{\bm{x}^{(k + 1)}} - \nabla \bm{f}\bp{\bm{x}^{(k)}}\right), \label{eq:dqn_joint_v_update} \\
    \bm{z}^{(k + 1)} &= W \bm{d}^{(k + 1)}, \label{eq:dqn_joint_z_update}
\end{align}
where $\bm{v}$ and $\bm{z}$ denote the concatenation of their respective variables, similar to $\bm{x}$, and $\bm{d}$ represents the concatenation of the local quasi-Newton step of all agents, with
\begin{align*}
    \bm{d}^{(k  + 1)} = \begin{bmatrix}
                    \llongdash & -\bp{C_{1}^{(k + 1)} v_{1}^{(k + 1)}}^{\T} & \rlongdash \\
                    & \vdots &  \\
                    \llongdash & -\bp{C_{N}^{(k + 1)} v_{N}^{(k + 1)}}^{\T}  & \rlongdash
                    \end{bmatrix},
\end{align*}
where $C_{i}$ is computed using a quasi-Newton update scheme, e.g., \eqref{eq:dis_BFGS_update}, \eqref{eq:dis_DFP_update}. Further, ${\bm{\alpha} = \diag(\alpha_{1},\dots,\alpha_{N})}$, where ${\alpha_{i} \in \mbb{R}_{++}}$ denotes the local step-size of agent $i$.
The update procedures of the DQN algorithm from the perspective of an individual agent is given by
\begin{align}
   {x}_{i}^{(k + 1)} &= \sum_{j \in \mcal{N}_{i}} w_{ij} \left({x}_{j}^{(k)} + {\alpha}_{j}^{(k)} {z}_{j}^{(k)}\right),  \label{eq:dqn_x_update} \\
   {z}_{i}^{(k + 1)} &= \sum_{j \in \mcal{N}_{i}} w_{ij} {d}_{j}^{(k + 1)}, \label{eq:dqn_z_update}
\end{align}
where ${{d}_{i}^{(k + 1)} = -C_{i}^{(k + 1)} v_{i}^{(k + 1)}}$, with the update procedure of $v_{i}$ given by \eqref{eq:dqn_v_update}.
DQN is initialized with ${x_{i}^{(0)} \in \mbb{R}^{n}}$, ${v_{i}^{(0)} = \nabla f_{i}(x_{i}^{(0)})}$, ${C_{i}^{(0)} \in \mbb{S}_{++}^{n}}$, and ${d_{i}^{(0)} = -C_{i}^{(0)} v_{i}^{(0)}}$, with $z_{i}^{(0)}$ computed via \eqref{eq:dqn_z_update} from ${d_{j}^{(0)}}$,~${\forall j \in \mcal{N}_{i}}$.
We outline the DQN method in Algorithm~\ref{alg:distributed_algorithm_uncons}.

\begin{algorithm2e} [th]
    \label{alg:distributed_algorithm_uncons}
    \caption{Distributed Quasi-Newton (DQN)}

    \SetKwRepeat{doparallel}{do in parallel}{while}

    \textbf{Initialization:} \\
    {\addtolength{\leftskip}{1em}
        ${x_{i}^{(0)} \in \mbb{R}^{n}}$, ${v_{i}^{(0)} = \nabla f_{i}(x_{i}^{(0)})}$, ${C_{i}^{(0)} \in \mbb{S}_{++}^{n}}$, ${d_{i}^{(0)} = -C_{i}^{(0)} v_{i}^{(0)}}$, and ${\alpha_{i}^{(0)} \in \mbb{R}_{++}}$,~${\forall i \in \mcal{V}}$.
    }

    \doparallel( $\forall i \in \mcal{V}$:){not converged or stopping criterion is not met}{
        $x_{i}^{(k + 1)} \leftarrow $ Procedure \eqref{eq:dqn_x_update} \algsep
        $v_{i}^{(k + 1)} \leftarrow $ Procedure \eqref{eq:dqn_v_update} \algsep
        \emph{Compute $C_{i}^{(k + 1)}$: e.g., \eqref{eq:dis_BFGS_update} or \eqref{eq:dis_DFP_update}.} \algsep
        $z_{i}^{(k + 1)} \leftarrow $ Procedure \eqref{eq:dqn_z_update} \algsep
        $k \leftarrow k + 1$
    }

\end{algorithm2e}

	Now, we analyze the convergence properties of the DQN algorithm, provided positive-definiteness of $C_{i}^{(k)}$ is maintained. 
Given the local variable $x_{i}$ of agent $i$, we define the disagreement between its local variable and the local variables of other agents by ${\tilde{x}_{i} = x_{i} - \frac{1}{N} \sum_{j \in \mcal{V}} x_{j}}$, representing the consensus error. We define the variables $\tilde{v}_{i}$ and $\tilde{z}_{i}$ similarly. In the subsequent discussion, we drop the argument of $\bm{g}$, denoting $\bm{g}(\bm{x}^{(k)})$ by $\bm{g}^{(k)}$, to simplify notation. Further, we denote the mean step-size at iteration $k$ by ${\mean{\bm{\alpha}}^{(k)} \in \mbb{R}^{N \times N}}$, a diagonal matrix with ${\mean{\bm{\alpha}}^{(k)}_{ii} = \frac{1}{N} \sum_{j \in \mcal{V}} \alpha_{j}^{(k)}}$,~${\forall i \in \mcal{V}}$, and, in addition, define ${r_{\alpha} = \alpha_{\max} \max_{k \geq 0} \frac{1}{\norm{\mean{\bm{\alpha}}^{(k)}}_{2}}}$, where ${\alpha_{\max} = \max_{k \geq 0} \left\{\norm{\bm{\alpha}^{(k)}}_{2} \right\}}$. We denote the spectral norm of the matrix $M$ associated with the communication network by $\lambda$. The following lemma shows convergence of the sequence ${\{\mean{\bm{x}}^{(k)}\}_{\forall k \geq 0}}$ to a limit point for sufficiently large $k$.

\begin{lemma}[Convergence of ${\{\mean{\bm{x}}^{(k)}\}_{\forall k \geq 0}}$]
    \label{lem:convergence_of_the_mean_sequence}
    Provided that ${\alpha_{\max} < \frac{1 - \lambda}{\lambda^{3} L (1 + r_{\alpha}) q}}$, the sequence ${\{\mean{\bm{x}}^{(k)}\}_{\forall k \geq 0}}$ converges to a limit point for sufficiently large $k$, with:
    \begin{equation}
        \lim_{k \rightarrow \infty} \norm{\mean{\bm{x}}^{(k + 1)} - \mean{\bm{x}}^{(k)}}_{2} = \lim_{k \rightarrow \infty} \norm{\mean{\bm{\alpha}^{(k)} \bm{z}^{(k)}}}_{2} = 0,
    \end{equation}
    where ${q^{2} = \gamma^{2} \min\{n, N\}}$.
    Further, as ${k \rightarrow \infty}$, the sequence ${\{\mean{\bm{g}}^{(k)}\}_{\forall k \geq 0}}$, denoting the average gradient of the objective function $f$, where each component of the gradient is evaluated at the local iterate of the corresponding agent, converges to zero, i.e.,
    \begin{equation}
        \label{eq:mean_gradient}
        \lim_{k \rightarrow \infty} \norm{\mean{\bm{g}}^{(k)}}_{2} = 0.
    \end{equation}
\end{lemma}

\begin{proof}
    The proof is provided in Appendix~A.
\end{proof}

Although Lemma~\ref{lem:convergence_of_the_mean_sequence} shows that 
the sequence ${\{\mean{\bm{x}}^{(k)}\}_{\forall k \geq 0}}$ has a limit point, Lemma~\ref{lem:convergence_of_the_mean_sequence} does not indicate that the sequence ${\{{x}_{i}^{(k)}\}_{\forall k \geq 0}}$ converges to a limit point,~${\forall i \in \mcal{V}}$. In the subsequent discussion, we show convergence of each agent's iterate to a limit point.

{
\color{black}
\begin{remark}
    In general, optimization methods based on Newton's method require a line-search algorithm for the selection of a suitable step-size for faster, guaranteed convergence. However, these methods are typically not amenable to distributed implementations. In line with prior work on distributed quasi-Newton methods \cite{eisen2017decentralized, eisen2019primal}, we do not consider the design of distributed line-search algorithms in this work. However, future work will seek to address this challenge. However, as demonstrated in prior work \cite{eisen2017decentralized, eisen2019primal}, in many situations, a constant step-size suffices, at the expense of a less-optimal convergence rate. We note that if a constant step-size is used, with ${\alpha < \frac{1 - \lambda}{2 \lambda^{3} L q}}$, then Lemma~\ref{lem:convergence_of_the_mean_sequence} is satisfied.
\end{remark}

}

\begin{corollary}[Convergence of ${\{\mean{\bm{v}}^{(k)}\}_{\forall k \geq 0}}$]
        \label{corr:mean_gradient_tracking}
       For sufficiently large $k$, the sequence ${\{\mean{\bm{v}}^{(k)}\}_{\forall k \geq 0}}$ converges to zero, with:
       \begin{equation}
           \label{eq:mean_gradient_tracking}
           \lim_{k \rightarrow \infty} \norm{\mean{\bm{v}}^{(k)}}_{2} = 0.
       \end{equation}
\end{corollary}

Corollary~\ref{corr:mean_gradient_tracking} follows from the fact that ${\mean{\bm{v}}^{(k)} = \mean{\bm{g}}^{(k)}}$,~${\forall k \geq 0}$. The limit in \eqref{eq:mean_gradient} does not conclusively indicate that the local iterates of all agents $\bm{x}^{(k + 1)}$ converge to a stationary point of the joint objective function $f$ in \eqref{eq:global_prob_uncons}. For this claim to be true, all agents must achieve agreement, computing a common solution for the optimization solution, which is shown in the following theorem.

\begin{theorem}[Consensus]
    \label{thm:convergence_to_the_mean}
    The disagreement errors of the local variables of all agents converge to zero, for sufficiently large $k$; i.e., the local variables of agent $i$, ${\left(x_{i}^{(k)}, v_{i}^{(k)}, z_{i}^{(k)}\right)}$, converge to the mean. Specifically:
    \begin{equation}
        \lim_{k \rightarrow \infty} \norm{\tilde{\bm{x}}^{(k)}}_{2} = 0,\ \lim_{k \rightarrow \infty} \norm{\tilde{\bm{v}}^{(k)}}_{2} = 0,\ \lim_{k \rightarrow \infty} \norm{\tilde{\bm{z}}^{(k)}}_{2} = 0.
    \end{equation}
\end{theorem}

\begin{proof}
    Please refer to Appendix~B.
\end{proof}

Together, Lemma~\ref{lem:convergence_of_the_mean_sequence} and Theorem~\ref{thm:convergence_to_the_mean} indicate that the local iterate of each agent,~${x_{i}^{(k)}}$, converges to a common solution $x^{(\infty)}$, as ${k \rightarrow \infty}$, given by the mean of the local iterates of all agents. Moreover, the limit point $x^{(\infty)}$ satisfies the \emph{first-order optimality conditions}.

\begin{theorem}[Convergence of the Objective Value]
    \label{thm:objective_convergence}
   The value of the aggregate objective function $\bm{f}$ converges to the optimal objective value $f^{\star}$, as ${k \rightarrow \infty}$, with:
   \begin{equation}
        \lim_{k \rightarrow \infty} \bm{f}(\bm{x}^{(k)}) = f^{\star}.
   \end{equation}
\end{theorem}

We provide a brief summary of the proof in 
Appendix~C,
which follows from convexity, coerciveness, and smoothness of $f$, enabling us to bound the optimality gap. Following from consensus among the agents, Theorem~\ref{thm:objective_convergence} indicates that the error between the optimal objective value of $f$ and the objective value evaluated at the local iterates of all agents $\bm{x}^{(k)}$ converges to zero.

	\section{Distributed Constrained Optimization}
\label{sec:distributed_alg_cons}
Two main approaches exist for solving the constrained problem in \eqref{eq:global_prob_cons}, namely: penalty-based approaches and Lagrangian-based approaches. Penalty-based approaches introduce the constraint function as an additional term in the objective function, with a penalty parameter determining the contribution of violations of the constraint function to the value of the composite objective function. However, these methods generally suffer from ill-conditioning in many situations. \emph{Exact} penalty-based methods overcome these challenges, but require optimization over non-smooth objective functions, making these methods difficult to implement in practice. In contrast, Lagrangian-based methods consider the Lagrangian associated with the problem in \eqref{eq:global_prob_cons}, computing a solution that satisfies the \emph{Karush-Kuhn-Tucker} (KKT) conditions. For the problem in \eqref{eq:global_prob_cons}, a point satisfying the KKT conditions associated with \eqref{eq:global_prob_cons} represents a global minimizer of \eqref{eq:global_prob_cons}.

In this work, we take a Lagrangian-based approach to solving \eqref{eq:global_prob_cons}. The Lagrangian $\mcal{L}$ of \eqref{eq:global_prob_cons} is given by:
\begin{equation}
	\label{eq:lagrangian}
	\begin{aligned}
		\mcal{L}(x, \beta) = \sum_{i = 1}^{N} f_{i}(x) + \beta^{\T} (Ax - b),
	\end{aligned}
\end{equation}
where ${\beta \in \mbb{R}^{m}}$ denotes the Lagrange multiplier associated with the equality constraint. A global minimizer $x^{\star}$ of \eqref{eq:global_prob_cons} satisfies the KKT conditions:
{
	\renewcommand*\labelenumi{C.\theenumi}
	\begin{enumerate}
		\item Stationarity Condition:
		\begin{equation}
			\label{eq:KKT_stationarity}
			\sum_{i = 1}^{N} \nabla f_{i}(x) + A^{\T}\beta = 0.
		\end{equation}

		\item Primal Feasibility:
		\begin{equation}
			\label{eq:KKT_primal_feasibility}
			Ax = b.
		\end{equation}
	\end{enumerate}
}
We can define a root-finding problem from the KKT conditions in \eqref{eq:KKT_stationarity} and \eqref{eq:KKT_primal_feasibility}, where we seek to compute a solution satisfying the system of equations given by:
\begin{equation}
	\label{eq:root_finding_problem}
	r(x,\beta) = 0,
\end{equation}
where $r$ denotes the residual computed at ${(x,\beta)}$, with:
\begin{equation}
	r(x,\beta) = \begin{bmatrix}
		\sum_{i = 1}^{N} \nabla f_{i}(x) + A^{\T}\beta \\
		Ax - b
	\end{bmatrix}.
\end{equation}
In general, the KKT conditions given by \eqref{eq:root_finding_problem} represents a nonlinear system of equations, which can be difficult to solve. However, in many cases, the root-finding problem can be solved efficiently using iterative methods, such as Newton's method, provided that $f$ is twice-differentiable. Given a current iterate ${(x^{(k)},\beta^{(k)})}$, Newton's method involves linearizing \eqref{eq:root_finding_problem} at the current iterate, yielding the system of linear equations:
\begin{equation}
	\label{eq:newton_update_direction}
	\begin{bmatrix}
		\sum_{i = 1}^{N} \nabla^{2} f_{i}(x^{(k)}) & A^{\T} \\
		A 													  & 0
	\end{bmatrix}
	\begin{bmatrix}
		\Delta x^{(k)}\\
		\Delta \beta^{(k)}
	\end{bmatrix}
	=
	- \begin{bmatrix}
	   		r_{\mathrm{stat}}^{(k)} \\
	   		r_{\mathrm{prim}}^{(k)}
	  \end{bmatrix},
\end{equation}
where ${(\Delta x^{(k)}, \Delta \beta^{(k)})}$ represents the update direction for the current iterate and:
\begin{equation}
	\begin{aligned}
		r_{\mathrm{stat}}^{(k)} &= \sum_{i = 1}^{N} \nabla f_{i}(x^{(k)}) + A^{\T}\beta^{(k)}, \\
		r_{\mathrm{prim}}^{(k)} &= Ax^{(k)} - b.
	\end{aligned}
\end{equation}
Subsequently, the next iterate is computed from:
\begin{equation}
	x^{(k + 1)} = x^{(k)} +\alpha^{(k)}  \Delta x^{(k)}, \quad \beta^{(k + 1)} = \beta^{(k)} + \alpha^{(k)} \Delta \beta^{(k)},
\end{equation}
where ${\alpha \in \mbb{R}_{++}}$ denotes the step-size.
Provided certain assumptions hold, we note that Newton's method is guaranteed to yield the global minimizer of \eqref{eq:global_prob_cons} \cite{deuflhard1979affine}.

In the multi-agent setting, notable challenges arise in directly implementing a Newton-like algorithm for equality-constrained optimization. Firstly, computation of the update direction in \eqref{eq:newton_update_direction} requires knowledge of the aggregate objective function (and the associated data) to compute the Hessian and the residual vector $r$, which is not accessible to any individual agent. Secondly, in many problems, computation of the Hessian of the objective function proves difficult, posing an additional challenge even if the aggregate objective function were to be known by all agents. In this work, we introduce a distributed algorithm \emph{EC-DQN}, derived from \emph{equality-constrained distributed quasi-Newton updates}, designed to address these challenges.

In our algorithm, each agent maintains a local copy of the problem variable $x$, representing its estimate of the solution of the aggregate problem \eqref{eq:global_prob_cons}, where ${x_{i} \in \mbb{R}^{n}}$ denotes agent $i$'s copy of $x$. To enable distributed computation of $r_{\mathrm{stat}}$ in \eqref{eq:newton_update_direction}, we utilize dynamic consensus techniques to enable each agent to compute an estimate of the gradient of $f$ locally. We denote agent $i$'s estimate of the average gradient of the local objective functions by ${v_{i} \in \mbb{R}^{n}}$, updated via the procedure:
\begin{equation}
	\label{eq:dynamic_avg_consensus_grad}
	v_{i}^{(k + 1)} = \sum_{j \in \mcal{N}_{i}} w_{ij} \left(v_{j}^{(k)} + \nabla f_{j} \left(x_{j}^{(k +1)}\right) - \nabla f_{j} \left(x_{j}^{(k)}\right) \right),
\end{equation}
at each iteration $k$. 

To circumvent the difficulty associated with computing the Hessian of the objective function, we utilize a quasi-Newton update scheme to estimate the Hessian of the objective function.
In this work, we utilize the DFP and BFGS update schemes \eqref{eq:dis_DFP_update} and \eqref{eq:dis_BFGS_update}, respectively.
Given its local estimate of the average gradient and the average Hessian, denoted by $B_{i}$, each agent computes an update direction for its local copy of $x$ from the system of linear equations:
\begin{equation}
	\label{eq:dis_newton_update_direction}
	\begin{bmatrix}
		B_{i}^{(k)} & A^{\T} \\
		A 													  & 0
	\end{bmatrix}
	\begin{bmatrix}
		\Delta x_{i}^{(k)} \\
		\beta_{i}^{(k + 1)}
	\end{bmatrix}
	=
	- \begin{bmatrix}
	   		r_{i, \mathrm{stat}}^{(k)} \\
	   		r_{i, \mathrm{prim}}^{(k)}
	  \end{bmatrix},
\end{equation}
noting the invariance of the optimal solution of \eqref{eq:global_prob_cons} to non-negative scaling of the objective function,
where:
\begin{equation}
	\begin{aligned}
		r_{i, \mathrm{stat}}^{(k)} &= v_{i}^{(k)}, \\
		r_{i, \mathrm{prim}}^{(k)} &= Ax^{(k)} - b.
	\end{aligned}
\end{equation}
In EC-DQN, we do not linearize the KKT conditions with respect to ${\beta_{i}}$, linearizing only with respect to $x_{i}$. Each agent maintains an auxiliary variable ${d_{i}^{(k)} \in \mbb{R}^{n}}$ associated with its update direction $\Delta x_{i}^{(k)}$, which is computed from:
\begin{equation}
	\label{eq:dis_fused_newton_update_direction}
	d_{i}^{(k)} = \sum_{j \in \mcal{N}_{i}} w_{ij} \Delta x_{j}^{(k)}.
\end{equation}
Subsequently, agent $i$ updates its local estimate of the solution of the optimization problem with:
\begin{equation}
	\label{eq:local_x_update}
	x_{i}^{(k + 1)} = \sum_{j \in \mcal{N}_{i}} w_{ij} \left(x_{j}^{(k)} + \alpha_{j} d_{j}^{(k)} \right),
\end{equation}
using the information received from its neighbors.
Algorithm~\ref{alg:distributed_algorithm_cons} outlines the update procedures of our distributed algorithm \emph{EC-DQN} for equality-constrained optimization.

\begin{algorithm2e} [th]
    \label{alg:distributed_algorithm_cons}
    \caption{EC-DQN: A Distributed Quasi-Newton Algorithm for Equality-Constrained Optimization}

    \SetKwRepeat{doparallel}{do in parallel}{while}

    \textbf{Initialization:} \\
    {\addtolength{\leftskip}{1em}
        ${x_{i}^{(0)} \in \mbb{R}^{n}}$, ${v_{i}^{(0)} = \nabla f_{i}\left(x_{i}^{(0)}\right)}$, ${B_{i}^{(0)} \in \mbb{S}_{++}^{n}}$, and ${\alpha_{i}^{(0)} \in \mbb{R}_{++}}$,~${\forall i \in \mcal{V}}$.
    }

    \doparallel( $\forall i \in \mcal{V}$:){not converged or stopping criterion is not met}{
        $\Delta x_{i}^{(k)} \leftarrow $ Procedure \eqref{eq:dis_newton_update_direction} \algsep
        $d_{i}^{(k)} \leftarrow $ Procedure \eqref{eq:dis_fused_newton_update_direction} \algsep
        $x_{i}^{(k + 1)} \leftarrow $ Procedure \eqref{eq:local_x_update} \algsep
        $v_{i}^{(k + 1)} \leftarrow $ Procedure \eqref{eq:dynamic_avg_consensus_grad} \algsep
        \emph{Compute $B_{i}^{(k + 1)}$.} \algsep
        $k \leftarrow k + 1$
    }

\end{algorithm2e}

We assume that that EC-DQN is initialized with a positive-definite estimate of the Hessian of the local objective function of each agent, i.e., ${B_{i}^{(0)} \in \mbb{S}_{++}^{n}}$,~${\forall i \in \mcal{V}}$. Further, we assume that each agent maintains a positive-definite approximation of the Hessian of its local objective function at each iteration, which may require modification of the quasi-Newton updates to preserve positive-definiteness of the estimates.

\begin{remark}
	In defining the update procedure of the local copy of $x$ maintained by each agent, we utilized the update direction computed by taking the weighted sum of each agent's corresponding solution of its local KKT system and those of its neighbors, denoted by $d_{i}^{(k)}$. We note that this additional procedure \eqref{eq:dis_fused_newton_update_direction} is not entirely necessary. In essence, agent $i$ could perform the update in \eqref{eq:local_x_update} using its local estimate $\Delta x_{i}^{(k)}$ in place of $d_{i}^{(k)}$. We observe empirically that introducing \eqref{eq:dis_fused_newton_update_direction} into EC-DQN does not provide significant improvements in its performance, except in sparsely-connected communication networks, where this additional communication step aids in enhancing the diffusion of information through the network of agents. However, we highlight that the improved performance in these situations comes at the expense of a greater communication overhead.
\end{remark}

\begin{theorem}
	The local iterates of all agents converge to the mean of the corresponding iterate. Specifically, the local iterate $x_{i}$ of agent converges to a limit point $x^{(\infty)}$ which satisfies the first-order optimality conditions,~${\forall i \in \mcal{V}}$.
\end{theorem}

We omit the proof here. The proof follows from 
convergence of the iterates ${(x_{i}^{(k)}, v_{i}^{(k)})}$ to the mean computed across all agents for sufficiently large $k$. In addition, convergence of quasi-Newton methods for equality-constrained optimization follows from \cite{fontecilla1987convergence, tapia1978quasi}.
	\section{Numerical Evaluations of DQN}
\label{sec:simulations_uncons}
We assess the performance of our distributed quasi-Newton method DQN on separable unconstrained optimization problems, comparing it to other notable distributed optimization methods, including the \emph{first-order} distributed algorithms DIGing-ATC \cite{nedic2017achieving}, C-ADMM \cite{mateos2010distributed}, and $ABm$ \cite{xin2019distributed} and the \emph{second-order} distributed algorithms \mbox{ESOM-$K$} \cite{mokhtari2016decentralized} and the distributed Newton's method D-Newton \cite{liu2023communication}. We note that the first-order methods do not require information on the curvature of the objective function (i.e., the Hessian of the objective function or its inverse), unlike the second-order methods. Further, we compare our method to the existing distributed quasi-Newton methods D-BFGS \cite{eisen2017decentralized} and \mbox{PD-QN} \cite{eisen2019primal}. In these evaluations, D-BFGS and \mbox{PD-QN} methods exhibit much slower convergence rates compared to DQN, requiring greater computation times along with greater communication costs. We present the detailed results of these evaluations in
Appendix~D.
We do not include the D-BFGS and \mbox{PD-QN} methods in this section, since in addition to not being competitive, these methods do not scale efficiently to larger problems.

We examine the convergence rate of each method in distributed quadratic programming problems and in logistic regression problems, across a range of communication networks with varying degrees of connectedness, characterized by the connectivity ratio ${\kappa = \frac{2 \vert \mcal{E} \vert}{N(N - 1)}}$, which represents the fraction of edges in the associated communication graph relative to the fully-connected case. We utilize the \emph{golden-section} method to select the optimal step-size for each method (and optimal penalty parameter in the case of C-ADMM). Further, we utilize \emph{Metropolis-Hastings} Weights in all the methods, with the exception of C-ADMM, which does not require a mixing matrix. We utilize doubly-stochastic weights in the $ABm$ method, and refer to this method as the $ABm$ or \mbox{${ABm}$-DS} method, interchangeably. In DQN, we utilize the DFP update scheme. We measure the convergence error using the \emph{relative-squared error} metric (RSE), associated with the optimal solution ${x^{\star} \in \mbb{R}^{n}}$, which is given by:
\begin{equation}
    \label{eq:rse_metric}
    \mathrm{RSE}(x_{i}) = \frac{\norm{x_{i} - x^{\star}}_{2}}{\norm{x^{\star}}_{2}},
\end{equation}
where ${x_{i} \in \mbb{R}^{n}}$ denotes the local iterate of agent $i$. We implement the algorithms on an AMD Ryzen $9$ $5900$X computer with $31$GB of RAM, measuring the computation time required by each algorithm to converge to the specified tolerance. We assume that the local copy of the problem variable maintained by each agent is represented using double-precision floating-point representation format. 
We examine the performance of each algorithm in logistic regression problems, discussed here, and quadratic programming problems, presented in the Supplementary Material.

\subsection{Logistic Regression}
\label{sec:sim_log_reg}
We examine the performance of each method in logistic regression problems, which arise in a variety of domains. We consider the logistic regression problem:
\begin{equation}
	\label{eq:dis_log_reg}
	\minimize{x \in \mathbb{R}^{n}} \frac{\xi}{2} \norm{x}_{2}^{2} + \sum_{i = 1}^{N } \sum_{j = 1}^{m_{i}} \ln \left(1 + \exp\left(-(a_{ij}^{\T}x)b_{ij} \right) \right),
\end{equation}
where ${\xi \in \mbb{R}_{+}}$ denotes the weight of the regularization term (added to prevent overfitting), and $m_{i}$ denotes the number of training samples ${\{(a_{ij}, b_{ij})\}_{j = 1}^{m_{i}}}$ available to agent $i$, with ${a_{ij} \in \mbb{R}^{n}}$ and ${b_{ij} \in \{-1, 1\}}$ representing the binary labels. We note that each agent has access to only its local training samples. We consider a problem with ${N = 50}$ agents, ${n = 40}$, and ${\xi = 1\mathrm{e}^{-2}}$. We randomly generate the training samples, with $m_{i}$ sampled from the uniform distribution over the interval ${[5, 30)}$.

We examine the performance of each algorithm on a randomly-generated connected communication graph with ${\kappa = 0.569}$. We initialize D-Newton \emph{Rank-$K$} with ${\theta = 0.03}$, ${M = 1}$, and ${K = 25}$, with the number of communication rounds for the multi-step consensus procedure set at $15$. Likewise, we initialize \mbox{ESOM-$K$} with ${\epsilon = 1}$ and ${K = 15}$. In DQN, we set the initial estimate of the inverse Hessian of each agent's objective function to $1\mathrm{e}^{-2} I_{n}$. We note that a closed-form solution does not exist for C-ADMM, limiting the competitiveness of C-ADMM. Consequently, we do not include C-ADMM in this evaluation. We provide the cumulative computation time and cumulative size of messages exchanged per agent in Table~\ref{tab:log_reg_all_alg}. While the first-order methods \mbox{$ABm$-DS} and DIGing-ATC fail to converge in this problem, DQN as well as the second-order methods D-Newton \emph{Rank-$K$} and ESOM converge to the optimal solution. Further, DQN attains the fastest computation time, while incurring the minimum communication cost. Despite using a generic guess as the initial estimate of the inverse Hessian in DQN, DQN still converges in about the same number of iterations as the second-order methods, as depicted in Figure~\ref{fig:log_reg_all_alg}. The first-order methods exhibit notably slower convergence rate in terms of the number of iterations compared to the other algorithms.

\begin{table}[th]
	\centering
	\caption{The cumulative computation time (in seconds), the cumulative size of messages exchanged per agent (in Megabytes), and the convergence status in the logistic regression problem, with ${\kappa = 0.569}$.}
	\label{tab:log_reg_all_alg}
	\begin{adjustbox}{width=\linewidth}
		{\begin{tabular}{l c c c}
				\toprule
				Algorithm & Computation Time (secs.) & Messages Exchanged (MB) & Converged \\
				\midrule
				$ABm$-DS \cite{xin2019distributed} & --- & --- & \xmark \\
				DIGing-ATC \cite{nedic2017achieving} & --- & --- & \xmark \\
				D-Newton \emph{Rank-$K$} \cite{liu2023communication} & $5.215 \mathrm{e}^{-1}$ & $7.392 \mathrm{e}^{0}$ & \cmark \\
				ESOM \cite{mokhtari2016decentralized} & $3.184 \mathrm{e}^{-2}$ & $9.683 \mathrm{e}^{-1}$ & \cmark \\
				DQN (ours) & $\bm{1.246 \mathrm{e}^{-2}}$ & $\bm{1.478 \mathrm{e}^{-1}}$ & \cmark \\
				\bottomrule
		\end{tabular}}
	\end{adjustbox}
\end{table}

\begin{figure}[th]
	\centering
	\includegraphics[width=\linewidth]{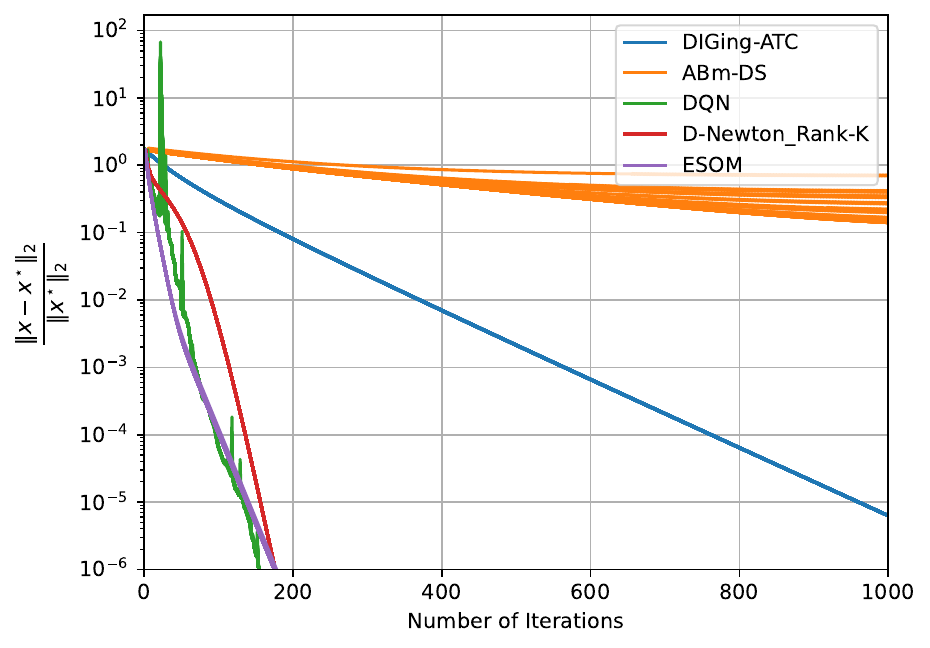}
	\caption{Per-iteration convergence error of each agent in the distributed logistic regression problem on a randomly-generated connected communication graph, with ${\kappa = 0.569}$. While the first-order methods DIGing-ATC and $ABm$-DS exhibit slow convergence, DQN and the second-order methods D-Newton \emph{Rank-$K$} and ESOM converge faster.}
	\label{fig:log_reg_all_alg}
\end{figure}

In addition, we assess the performance across a range of communication networks, over $20$ logistic regression problems in Table~\ref{tab:log_reg_sweep_success}. We initialize each algorithm with the same parameters, retaining ESOM while dropping D-Newton \emph{Rank-$K$} in the evaluations.  While ESOM achieves the fastest computation time when ${\kappa = 0.21}$, DQN achieves the fastest computation time for all other values of $\kappa$, compared to the other algorithms. Moreover, DQN and ESOM attain a perfect success rate, unlike the first-order methods. 
DQN incurs the least communication cost for convergence, even when each algorithm has an unlimited number of iterations available for convergence. Compared to ESOM, DQN incurs about an order of magnitude lower communication cost for all values of $\kappa$, with the exception of ${\kappa = 0.21}$ where DQN reduces the communication cost by a factor of about four, relative to ESOM.

\begin{table*}[t]
	\centering
	\caption{The mean and standard deviation of the cumulative computation time (Comp. Time), in seconds (secs.), and the cumulative size of messages exchanged per agent (Comm. Cost), in Megabytes (MB), in $20$ {logistic regression} problems, across a range of communication networks, in addition to the success rate of each algorithm in converging to the optimal solution. Entries associated with a non-zero success rate below $100\%$ are displayed in red.}
	\label{tab:log_reg_sweep_success}
	\begin{adjustbox}{width=\linewidth}
		{\begin{tabular}{l c c c c c c}
				\toprule
				\multirow{2}{*}{\textbf{Algorithm}} & \multicolumn{3}{c}{$\bm{\kappa} = \bm{0.21}$} & \multicolumn{3}{c}{$\bm{\kappa} = \bm{0.57}$} \\
				\cmidrule(lr){2-4} \cmidrule(lr){5-7}
				 & Comp. Time (secs.) & Comm. Cost (MB) & Success Rate ($\%$)
				 & Comp. Time (secs.) & Comm. Cost (MB) & Success Rate ($\%$) \\
				\midrule
				$ABm$-DS \cite{xin2019distributed} & --- & --- & $0$
				& --- & --- & $0$ \\
				DIGing-ATC \cite{nedic2017achieving} & {\color{red} $3.25\mathrm{e}^{-2} \pm 8.56\mathrm{e}^{-3}$} & {\color{red} $5.62\mathrm{e}^{-1} \pm 7.55\mathrm{e}^{-2}$} & {\color{red} $40$}
				& {\color{red} $2.89\mathrm{e}^{-2}  \pm 6.19\mathrm{e}^{-3}$} & {\color{red} $5.02\mathrm{e}^{-1} \pm 8.24\mathrm{e}^{-2}$} & {\color{red} $85$} \\
				ESOM \cite{mokhtari2016decentralized} & $\bm{2.60\mathrm{e}^{-2}  \pm 1.69\mathrm{e}^{-2}}$ & $1.17\mathrm{e}^{0} \pm 8.09\mathrm{e}^{-1}$ & $100$
				& $2.31\mathrm{e}^{-2}  \pm 1.44\mathrm{e}^{-2}$ & $1.16\mathrm{e}^{0} \pm 8.59\mathrm{e}^{-1}$ & $100$ \\
				DQN (ours) & $3.91\mathrm{e}^{-2}  \pm 1.40\mathrm{e}^{-2}$ & $\bm{2.72\mathrm{e}^{-1} \pm 1.28\mathrm{e}^{-1}}$ & $100$
				& $\bm{1.84\mathrm{e}^{-2}  \pm 5.19\mathrm{e}^{-3}}$ & $\bm{1.18\mathrm{e}^{-1} \pm 3.85\mathrm{e}^{-2}}$ & $100$ \\
				\bottomrule
		\end{tabular}}
	\end{adjustbox}

	\bigskip

	\begin{adjustbox} {width=\linewidth}
		{\begin{tabular}{l c c c c c c}
				\toprule
				\multirow{2}{*}{\textbf{Algorithm}} & \multicolumn{3}{c}{$\bm{\kappa} = \bm{0.74}$} & \multicolumn{3}{c}{$\bm{\kappa} = \bm{0.85}$} \\
				\cmidrule(lr){2-4} \cmidrule(lr){5-7}
		 & Comp. Time (secs.) & Comm. Cost (MB) & Success Rate ($\%$)
		 & Comp. Time (secs.) & Comm. Cost (MB) & Success Rate ($\%$) \\
				\midrule
				$ABm$-DS \cite{xin2019distributed} & --- & --- & $0$
				& --- & --- & $0$ \\
				DIGing-ATC \cite{nedic2017achieving} & {\color{red} $2.68\mathrm{e}^{-2}  \pm 8.49\mathrm{e}^{-3}$} & {\color{red} $4.99\mathrm{e}^{-1} \pm 1.10\mathrm{e}^{-1}$} & {\color{red} $95$}
				& {\color{red} $2.48\mathrm{e}^{-2}  \pm 7.31\mathrm{e}^{-3}$} & {\color{red} $4.64\mathrm{e}^{-1} \pm 1.18\mathrm{e}^{-1}$} & {\color{red} $85$} \\
				ESOM \cite{mokhtari2016decentralized}& $2.42\mathrm{e}^{-2}  \pm 1.42\mathrm{e}^{-2}$ & $1.16\mathrm{e}^{0} \pm 7.78\mathrm{e}^{-1}$ & $100$
				& $2.33\mathrm{e}^{-2}  \pm 1.38\mathrm{e}^{-2}$ & $1.15\mathrm{e}^{0} \pm 7.62\mathrm{e}^{-1}$ & $100$ \\
				DQN (ours) & $\bm{1.65\mathrm{e}^{-2}  \pm 4.85\mathrm{e}^{-3}}$ & $\bm{1.06\mathrm{e}^{-1} \pm 1.60\mathrm{e}^{-2}}$ & $100$
				& $\bm{1.63\mathrm{e}^{-2}  \pm 6.13\mathrm{e}^{-3}}$ & $\bm{1.04\mathrm{e}^{-1} \pm 2.35\mathrm{e}^{-2}}$ & $100$ \\
				\bottomrule
		\end{tabular}}
	\end{adjustbox}
\end{table*}

	\section{Numerical Evaluations of EC-DQN}
\label{sec:simulations_cons}
We evaluate the performance of EC-DQN, in comparison to existing distributed algorithms for multi-agent constrained optimization, including consensus ADMM (C-ADMM) \cite{mateos2010distributed}, SONATA \cite{sun2022distributed}, and DPDA \cite{aybat2016primal}. In our evaluations, we consider equality-constrained \emph{basis pursuit denoising} and \emph{logistic regression problems} across communication networks with different \emph{connectivity ratios} ${\kappa}$. 
We utilize the same evaluation setup as in Section~\ref{sec:simulations_uncons}.
The update procedures in C-ADMM and SONATA require each agent to solve an optimization problem, adversely affecting the computation time of each agent if these problems are solved exactly. To ensure that these methods remain competitive to other algorithms with respect to the computation time, we solve the pertinent optimization problems inexactly using sequential quadratic programming techniques. Empirical results indicate that the inexact updates did not have a material effect on the convergence rates of these algorithms.
We provide additional results in the Supplementary Material.

\subsection{Basis Pursuit Denoising}
We consider the following optimization problem:
\begin{equation}
	\label{eq:basis_pursuit}
	\begin{aligned}
		\minimize{x \in \mbb{R}^{n}} &\sum_{i = 1}^{N} \frac{1}{2} \norm{A_{i}x - b_{i}}_{2}^{2} + \xi \norm{x}_{1} \\
		\subj &Fx = e,
	\end{aligned}
\end{equation}
where ${A_{i} \in \mbb{R}^{p_{i} \times n}}$, ${b_{i} \in \mbb{R}^{p_{i}}}$ represent the local data only available to agent $i$. Moreover, each agent has access to ${(F \in \mbb{R}^{m \times n}, e \in \mbb{R}^{m})}$ defining the constraints in the optimization problem. The parameter $\xi$ determines the relative weighting between the reconstruction quality associated with the $\ell_{2}^{2}$-norm and sparsity of the solution associated with the \mbox{$\ell_{1}$-norm}. Problems of this form arise in a variety of applications such as image compression, compressed sensing, and least absolute shrinkage and selection operator (LASSO) regression. Similar problems also arise in \emph{sparse dictionary learning}. We note that the basis-pursuit problem is non-smooth, and as such, the assumptions on smoothness made in this paper do not apply. Nonetheless, we demonstrate that our algorithm is effective in subdifferentiable problems, even though these problems might not be smooth. We consider a multi-agent setting with ${N = 50}$ agents, ${n = 40}$, and ${\xi = 1\mathrm{e}^{-2}}$. We randomly generate the problem data available to each agent, including the number of constraints, such that the rank of ${A_{i} < n}$,~${\forall i \in \mcal{V}}$. We execute each algorithm for a maximum of $1000$ iterations, with a convergence threshold of $1\mathrm{e}^{-8}$. We initialize each agent's estimate of the Hessian of the problem to a randomly-generated positive-definite symmetric matrix in EC-DQN.

We 
examine the performance of each algorithm in \emph{poorly-conditioned} basis pursuit denoising problems, with the condition number ranging between $53.895$ and $125.029$, on randomly-generated communication networks with different connectivity ratios. Table~\ref{tab:poorly_conditioned_basis_run} summarizes the performance of each algorithm in these problems, highlighting the cumulative computation time and communication cost on $20$ instances of such problems. Similar to the results obtained in the well-conditioned problem setting, when ${\kappa = 0.22}$, C-ADMM attains the least communication cost, with SONATA achieving the fastest computation time for convergence. Although at ${\kappa = 0.22}$, we note that all algorithms achieve similar performance in terms of computation time and communication cost. For all other values of $\kappa$, EC-DQN converges the fastest, while requiring the least communication overhead. Further, the performance of DPDA degrades notably, as illustrated in Figure~\ref{fig:poorly_conditioned_basis_run}. In Figure~\ref{fig:poorly_conditioned_basis_run}, we show the performance of each algorithm on an instance of the problems with ${\kappa = 0.566}$ and a condition number of $99.698$. In this problem, EC-DQN and C-ADMM incur the least communication cost, requiring $5.19\mathrm{e}^{-2}$~MB for convergence. However, EC-DQN attains the fastest computation time of $1.43\mathrm{e}^{-2}$~secs., compared to the second-fastest method SONATA, which requires $2.33\mathrm{e}^{-2}$~secs. and a communication cost of $9.28\mathrm{e}^{-2}$~MB. C-ADMM requires $2.51\mathrm{e}^{-2}$ secs. for convergence. Further, we note that \mbox{C-ADMM}, SONATA, and EC-DQN converged on all problem instances, in the well-conditioned and poorly-conditioned settings.

\begin{table*}[t]
	\centering
	\caption{The mean and standard deviation of the cumulative computation time (Comp. Time), in seconds (secs.), and communication cost (Comm. Cost), in Megabytes (MB), per agent, in $20$ poorly-conditioned basis pursuit denoising problems, on different communication networks.}
	\label{tab:poorly_conditioned_basis_run}
		{\begin{tabular}{l c c c c}
				\toprule
				\multirow{2}{*}{\textbf{Algorithm}} & \multicolumn{2}{c}{$\bm{\kappa} = \bm{0.22}$} & \multicolumn{2}{c}{$\bm{\kappa} = \bm{0.57}$} \\
				\cmidrule(lr){2-3} \cmidrule(lr){4-5}
				 & Comp. Time (secs.) & Comm. Cost (MB)
				 & Comp. Time (secs.) & Comm. Cost (MB) \\
				\midrule
				C-ADMM \cite{mateos2010distributed} & $4.67\mathrm{e}^{-2}  \pm 1.28\mathrm{e}^{-2}$ & $\bm{1.35\mathrm{e}^{-1}  \pm 3.09\mathrm{e}^{-2}}$
				& $4.33\mathrm{e}^{-2}  \pm 1.53\mathrm{e}^{-2}$ & $1.27\mathrm{e}^{-1} \pm 4.05\mathrm{e}^{-2}$ \\
				DPDA \cite{aybat2016primal} & --- & ---
				& --- & --- \\
				SONATA \cite{sun2022distributed} & $\bm{2.85\mathrm{e}^{-2}  \pm 1.09\mathrm{e}^{-2}}$ & $1.76\mathrm{e}^{-1} \pm 7.01\mathrm{e}^{-2}$
				& $1.73\mathrm{e}^{-2}  \pm 7.70\mathrm{e}^{-3}$ & $1.34\mathrm{e}^{-1} \pm 7.50\mathrm{e}^{-2}$ \\
				EC-DQN (ours) & $2.88\mathrm{e}^{-2}  \pm 8.44\mathrm{e}^{-3}$ & $1.89\mathrm{e}^{-1} \pm 3.65\mathrm{e}^{-2}$
				& $\bm{9.99\mathrm{e}^{-3}  \pm 2.21\mathrm{e}^{-3}}$ & $\bm{6.92\mathrm{e}^{-2} \pm 2.26\mathrm{e}^{-2}}$ \\
				\bottomrule
		\end{tabular}}

	\bigskip

		{\begin{tabular}{l c c c c}
				\toprule
				\multirow{2}{*}{\textbf{Algorithm}} & \multicolumn{2}{c}{$\bm{\kappa} = \bm{0.75}$} & \multicolumn{2}{c}{$\bm{\kappa} = \bm{0.86}$} \\
				\cmidrule(lr){2-3} \cmidrule(lr){4-5}
		 & Comp. Time (secs.) & Comm. Cost (MB)
		 & Comp. Time (secs.) & Comm. Cost (MB) \\
				\midrule
				C-ADMM \cite{mateos2010distributed} & $4.03\mathrm{e}^{-2}  \pm 1.56\mathrm{e}^{-2}$ & $1.21\mathrm{e}^{-1}  \pm 4.26\mathrm{e}^{-2}$
				& $4.24\mathrm{e}^{-2}  \pm 1.82\mathrm{e}^{-2}$ & $1.22\mathrm{e}^{-1} \pm4.56\mathrm{e}^{-2}$ \\
				DPDA \cite{aybat2016primal} & --- & ---
				& --- & --- \\
				SONATA \cite{sun2022distributed} & $1.44\mathrm{e}^{-2}  \pm 5.35\mathrm{e}^{-3}$ & $1.11\mathrm{e}^{-1} \pm 6.60\mathrm{e}^{-2}$
				& $1.36\mathrm{e}^{-2}  \pm 6.95\mathrm{e}^{-3}$ & $1.16\mathrm{e}^{-1} \pm 7.94\mathrm{e}^{-2}$ \\
				EC-DQN (ours) & $\bm{7.36\mathrm{e}^{-3}  \pm 1.67\mathrm{e}^{-3}}$ & $\bm{5.25\mathrm{e}^{-2} \pm 1.83\mathrm{e}^{-2}}$
				& $\bm{7.51\mathrm{e}^{-3}  \pm 3.14\mathrm{e}^{-3}}$ & $\bm{5.32\mathrm{e}^{-2} \pm 1.86\mathrm{e}^{-2}}$ \\
				\bottomrule
		\end{tabular}}
\end{table*}

\begin{figure}[th]
	\centering
	\includegraphics[width=\linewidth]{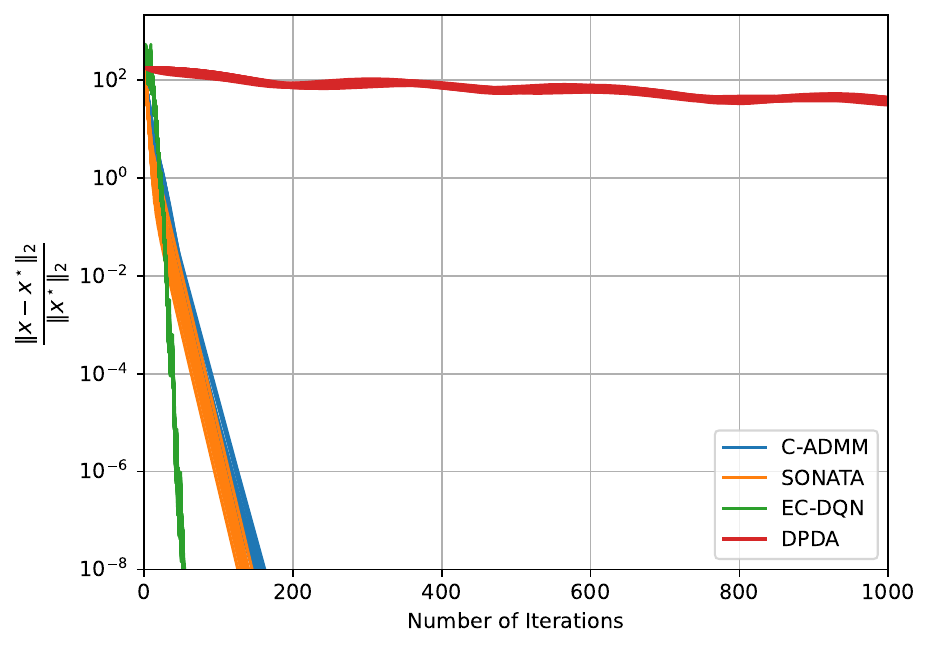}
	\caption{Convergence error of each algorithm in a poorly-conditioned basis pursuit denoising problem, on a randomly-generated communication graph with ${\kappa = 0.566}$. EC-DQN achieves the fastest convergence rate compared to the other algorithms.}
	\label{fig:poorly_conditioned_basis_run}
\end{figure}

\subsection{Logistic Regression}
In addition, we evaluate each algorithm in constrained logistic regression problems, expressed in the form:
\begin{equation}
	\label{eq:eval_log_reg}
	\begin{aligned}
		 \minimize{x \in \mathbb{R}^{n}} &\frac{\xi}{2} \norm{x}_{2}^{2} + \sum_{i = 1}^{N } \sum_{j = 1}^{m_{i}} \ln \left(1 + \exp\left(-(a_{ij}^{\T}x)b_{ij} \right) \right),\\
		\subj &Fx = e,
	\end{aligned}
\end{equation}
where ${\{(a_{ij}, b_{ij})\}_{j = 1}^{m_{i}}}$ denotes the training data available to agent $i$ and $m_{i}$ denotes the number of training data samples, with regularization parameter ${\xi \in \mbb{R}_{+}}$, ${F \in \mbb{R}^{m \times n}}$, and ${e \in \mbb{R}^{m}}$. We consider the problem in \eqref{eq:eval_log_reg} with ${N = 50}$ agents and ${n = 40}$. Further, we randomly generate the problem data, the communication networks, and each agent's estimate of the Hessian of the problem. We utilize a convergence threshold of $1\mathrm{e}^{-7}$, with  ${\xi = 1\mathrm{e}^{-2}}$ and the maximum number of iterations set at $1000$.

In Table~\ref{tab:eval_logistic_regression_stats}, we summarize the performance of each algorithm on $20$ instances of the logistic regression problem, noting the total computation time and communication cost per agent required by each algorithm for convergence on different communication graphs. As presented in Table~\ref{tab:eval_logistic_regression_stats}, EC-DQN achieves the fastest computation time and least communication cost across all communication networks, except at ${\kappa = 0.22}$ where C-ADMM attains the minimum communication cost. We note that, on each communication graph, C-ADMM failed to converge in one problem, whereas EC-DQN and SONATA converged in all cases. When an algorithm fails to converge, we do not utilize the computation time and communication cost for that run in computing the summary statistics for that algorithm. Figure~\ref{fig:eval_logistic_regression_stats} depicts the convergence rate of each algorithm per iteration on an instance of \eqref{eq:eval_log_reg} with ${\kappa = 0.566}$. We note that EC-DQN converges within the fewest number of iterations. Moreover, EC-DQN requires the least communication cost of ${2.50\mathrm{e}^{-2}}$~MB with the fastest computation time ($8.00\mathrm{e}^{-3}$~secs.) for convergence. In this problem, SONATA requires a computation time and communication cost of $2.49\mathrm{e}^{-2}$~secs. and $6.15\mathrm{e}^{-2}$~MB, respectively, outperforming C-ADMM, which requires a computation time and communication cost of $4.49\mathrm{e}^{-2}$~secs. and $6.56\mathrm{e}^{-2}$~MB, respectively. DPDA fails to converge within $1000$ iterations.

\begin{table*}[t]
	\centering
	\caption{The mean and standard deviation of the cumulative computation time (Comp. Time), in seconds (secs.), and communication cost (Comm. Cost), in Megabytes (MB), per agent, in $20$ constrained logistic regression problems, on different communication networks.}
	\label{tab:eval_logistic_regression_stats}
		{\begin{tabular}{l c c c c}
				\toprule
				\multirow{2}{*}{\textbf{Algorithm}} & \multicolumn{2}{c}{$\bm{\kappa} = \bm{0.22}$} & \multicolumn{2}{c}{$\bm{\kappa} = \bm{0.57}$} \\
				\cmidrule(lr){2-3} \cmidrule(lr){4-5}
				 & Comp. Time (secs.) & Comm. Cost (MB)
				 & Comp. Time (secs.) & Comm. Cost (MB) \\
				\midrule
				C-ADMM \cite{mateos2010distributed} & $4.05\mathrm{e}^{-2}  \pm 2.74\mathrm{e}^{-2}$ & $\bm{7.86\mathrm{e}^{-2}  \pm 4.31\mathrm{e}^{-2}}$
				& $3.95\mathrm{e}^{-2}  \pm 2.06\mathrm{e}^{-2}$ & $7.70\mathrm{e}^{-2} \pm 4.25\mathrm{e}^{-2}$ \\
				DPDA \cite{aybat2016primal} & --- & ---
				& --- & --- \\
				SONATA \cite{sun2022distributed} & $3.81\mathrm{e}^{-2}  \pm 1.64\mathrm{e}^{-2}$ & $1.80\mathrm{e}^{-1} \pm 1.17\mathrm{e}^{-1}$
				& $2.88\mathrm{e}^{-2}  \pm 1.22\mathrm{e}^{-2}$ & $1.39\mathrm{e}^{-1} \pm 8.82\mathrm{e}^{-2}$ \\
				EC-DQN (ours) & $\bm{2.41\mathrm{e}^{-2}  \pm 1.16\mathrm{e}^{-2}}$ & $1.43\mathrm{e}^{-1} \pm 4.43\mathrm{e}^{-2}$
				& $\bm{5.79\mathrm{e}^{-3}  \pm 3.89\mathrm{e}^{-3}}$ & $\bm{4.13\mathrm{e}^{-2} \pm 3.37\mathrm{e}^{-2}}$ \\
				\bottomrule
		\end{tabular}}

	\bigskip

		{\begin{tabular}{l c c c c}
				\toprule
				\multirow{2}{*}{\textbf{Algorithm}} & \multicolumn{2}{c}{$\bm{\kappa} = \bm{0.75}$} & \multicolumn{2}{c}{$\bm{\kappa} = \bm{0.86}$} \\
				\cmidrule(lr){2-3} \cmidrule(lr){4-5}
		 & Comp. Time (secs.) & Comm. Cost (MB)
		 & Comp. Time (secs.) & Comm. Cost (MB) \\
				\midrule
				C-ADMM \cite{mateos2010distributed} & $4.03\mathrm{e}^{-2}  \pm 2.19\mathrm{e}^{-2}$ & $7.95\mathrm{e}^{-2}  \pm 4.71\mathrm{e}^{-2}$
				& $4.51\mathrm{e}^{-2}  \pm 2.41\mathrm{e}^{-2}$ & $7.90\mathrm{e}^{-2} \pm 3.84\mathrm{e}^{-2}$ \\
				DPDA \cite{aybat2016primal} & --- & ---
				& --- & --- \\
				SONATA \cite{sun2022distributed} & $2.80\mathrm{e}^{-2}  \pm 1.44\mathrm{e}^{-2}$ & $1.37\mathrm{e}^{-1} \pm 1.01\mathrm{e}^{-1}$
				& $2.51\mathrm{e}^{-2}  \pm 1.24\mathrm{e}^{-2}$ & $1.23\mathrm{e}^{-1} \pm 7.85\mathrm{e}^{-2}$ \\
				EC-DQN (ours) & $\bm{5.57\mathrm{e}^{-3}  \pm 4.14\mathrm{e}^{-3}}$ & $\bm{3.66\mathrm{e}^{-2} \pm 3.58\mathrm{e}^{-2}}$
				& $\bm{4.95\mathrm{e}^{-3}  \pm 3.08\mathrm{e}^{-3}}$ & $\bm{2.89\mathrm{e}^{-2} \pm 2.12\mathrm{e}^{-2}}$ \\
				\bottomrule
		\end{tabular}}
\end{table*}

\begin{figure}[th]
	\centering
	\includegraphics[width=\linewidth]{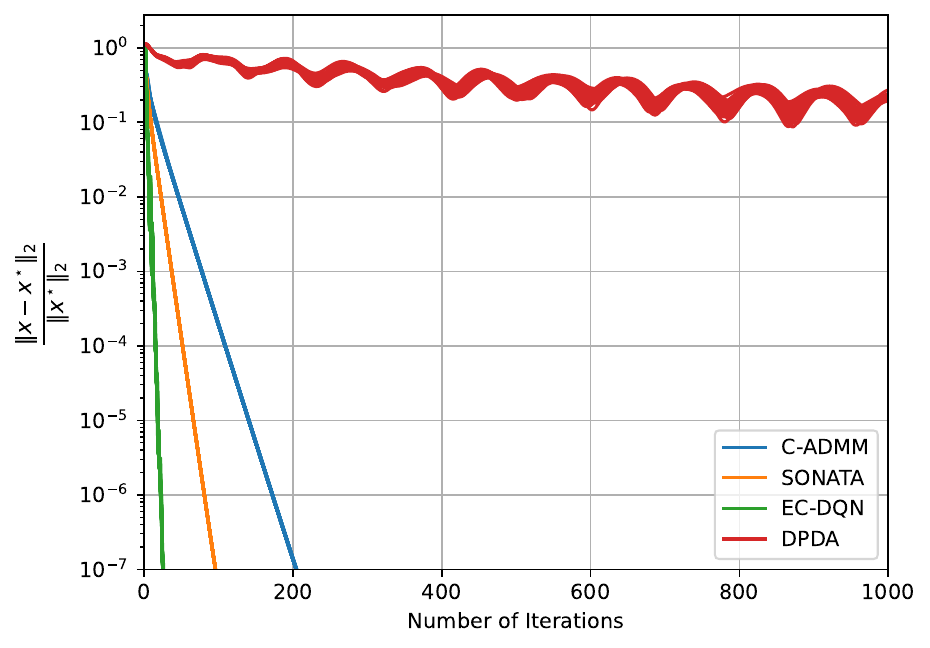}
	\caption{Convergence error of each algorithm in constrained logistic regression problems with ${\kappa = 0.566}$. EC-DQN converges the fastest in comparison to the other algorithms.}
	\label{fig:eval_logistic_regression_stats}
\end{figure}
	\section{Conclusion}
\label{sec:conclusion}
We introduce distributed quasi-Newton algorithms: DQN---for unconstrained optimization---and EC-DQN---for equality-constrained optimization---that enable a group of agents to solve multi-agent optimization problems using an estimate of the curvature of the objective function of the problem. In our algorithms, each agent leverages information on the curvature of the objective function using its local estimate of the aggregate gradient, while communicating with its immediate, one-hop neighbors over a communication network. Our algorithms converge to a stationary point of the optimization problem and provide faster empirical convergence in unconstrained and constrained problems, particularly in ill-conditioned problem settings, where distributed first-order methods show slow convergence rates. Moreover, DQN avoids the significant computation and communication overhead associated with existing distributed second-order and quasi-Newton methods by eschewing exact computation of the Hessian of the problem, in addition to estimating the inverse Hessian directly, making it amenable to problems where computation of the Hessian of the objective function is impractical or impossible. Future work will examine the extension of the proposed algorithms to convex inequality-constrained optimization problems, and more generally, non-convex constrained optimization problems, which arise in many applications. Further, future work will seek to design distributed algorithms for selecting an optimal step-size for convergence and provide additional theoretical analysis of the convergence rates of our algorithms.

	\bibliographystyle{style/IEEEtran}
	\bibliography{references}

\begin{thebibliography}{10}
\providecommand{\url}[1]{#1}
\csname url@samestyle\endcsname
\providecommand{\newblock}{\relax}
\providecommand{\bibinfo}[2]{#2}
\providecommand{\BIBentrySTDinterwordspacing}{\spaceskip=0pt\relax}
\providecommand{\BIBentryALTinterwordstretchfactor}{4}
\providecommand{\BIBentryALTinterwordspacing}{\spaceskip=\fontdimen2\font plus
\BIBentryALTinterwordstretchfactor\fontdimen3\font minus \fontdimen4\font\relax}
\providecommand{\BIBforeignlanguage}[2]{{%
\expandafter\ifx\csname l@#1\endcsname\relax
\typeout{** WARNING: IEEEtran.bst: No hyphenation pattern has been}%
\typeout{** loaded for the language `#1'. Using the pattern for}%
\typeout{** the default language instead.}%
\else
\language=\csname l@#1\endcsname
\fi
#2}}
\providecommand{\BIBdecl}{\relax}
\BIBdecl

\bibitem{bashir2021aerodynamic}
M.~Bashir, S.~Longtin-Martel, R.~M. Botez, and T.~Wong, ``Aerodynamic design optimization of a morphing leading edge and trailing edge airfoil--application on the uas-s45,'' \emph{Applied Sciences}, vol.~11, no.~4, p. 1664, 2021.

\bibitem{masdari2019optimization}
M.~Masdari, M.~Tahani, M.~H. Naderi, and N.~Babayan, ``Optimization of airfoil based savonius wind turbine using coupled discrete vortex method and salp swarm algorithm,'' \emph{Journal of Cleaner Production}, vol. 222, pp. 47--56, 2019.

\bibitem{harish2016reduced}
V.~Harish and A.~Kumar, ``Reduced order modeling and parameter identification of a building energy system model through an optimization routine,'' \emph{Applied Energy}, vol. 162, pp. 1010--1023, 2016.

\bibitem{ahmadianfar2021gradient}
I.~Ahmadianfar, W.~Gong, A.~A. Heidari, N.~A. Golilarz, A.~Samadi-Koucheksaraee, and H.~Chen, ``Gradient-based optimization with ranking mechanisms for parameter identification of photovoltaic systems,'' \emph{Energy Reports}, vol.~7, pp. 3979--3997, 2021.

\bibitem{long2021parameters}
W.~Long, T.~Wu, M.~Xu, M.~Tang, and S.~Cai, ``Parameters identification of photovoltaic models by using an enhanced adaptive butterfly optimization algorithm,'' \emph{Energy}, vol. 229, p. 120750, 2021.

\bibitem{toumieh2022decentralized}
C.~Toumieh and A.~Lambert, ``Decentralized multi-agent planning using model predictive control and time-aware safe corridors,'' \emph{IEEE Robotics and Automation Letters}, vol.~7, no.~4, pp. 11\,110--11\,117, 2022.

\bibitem{torreno2014fmap}
A.~Torreno, E.~Onaindia, and O.~Sapena, ``Fmap: Distributed cooperative multi-agent planning,'' \emph{Applied Intelligence}, vol.~41, pp. 606--626, 2014.

\bibitem{mishra2019multi}
S.~Mishra, C.~Bordin, A.~Tomasgard, and I.~Palu, ``A multi-agent system approach for optimal microgrid expansion planning under uncertainty,'' \emph{International Journal of Electrical Power \& Energy Systems}, vol. 109, pp. 696--709, 2019.

\bibitem{franze2018distributed}
G.~Franz{\`e}, W.~Lucia, and F.~Tedesco, ``A distributed model predictive control scheme for leader--follower multi-agent systems,'' \emph{International Journal of Control}, vol.~91, no.~2, pp. 369--382, 2018.

\bibitem{wang2014synthesis}
P.~Wang and B.~Ding, ``A synthesis approach of distributed model predictive control for homogeneous multi-agent system with collision avoidance,'' \emph{International Journal of Control}, vol.~87, no.~1, pp. 52--63, 2014.

\bibitem{luo2017multi}
R.~Luo, R.~Bourdais, T.~J. van~den Boom, and B.~De~Schutter, ``Multi-agent model predictive control based on resource allocation coordination for a class of hybrid systems with limited information sharing,'' \emph{Engineering Applications of Artificial Intelligence}, vol.~58, pp. 123--133, 2017.

\bibitem{eisen2017decentralized}
M.~Eisen, A.~Mokhtari, and A.~Ribeiro, ``Decentralized quasi-newton methods,'' \emph{IEEE Transactions on Signal Processing}, vol.~65, no.~10, pp. 2613--2628, 2017.

\bibitem{eisen2019primal}
------, ``A primal-dual quasi-newton method for exact consensus optimization,'' \emph{IEEE Transactions on Signal Processing}, vol.~67, no.~23, pp. 5983--5997, 2019.

\bibitem{yang2019survey}
T.~Yang, X.~Yi, J.~Wu, Y.~Yuan, D.~Wu, Z.~Meng, Y.~Hong, H.~Wang, Z.~Lin, and K.~H. Johansson, ``A survey of distributed optimization,'' \emph{Annual Reviews in Control}, vol.~47, pp. 278--305, 2019.

\bibitem{shorinwa2023distributedsurvey}
O.~Shorinwa, T.~Halsted, J.~Yu, and M.~Schwager, ``Distributed optimization methods for multi-robot systems: Part ii--a survey,'' \emph{arXiv preprint arXiv:2301.11361}, 2023.

\bibitem{nedic2009}
A.~Nedic and A.~Ozdaglar, ``Distributed subgradient methods for multi-agent optimization,'' \emph{IEEE Transactions on Automatic Control}, vol.~54, no.~1, pp. 48--61, 2009.

\bibitem{lobel2010distributed}
I.~Lobel and A.~Ozdaglar, ``Distributed subgradient methods for convex optimization over random networks,'' \emph{IEEE Transactions on Automatic Control}, vol.~56, no.~6, pp. 1291--1306, 2010.

\bibitem{shi2015extra}
W.~Shi, Q.~Ling, G.~Wu, and W.~Yin, ``{EXTRA}: An exact first-order algorithm for decentralized consensus optimization,'' \emph{SIAM Journal on Optimization}, vol.~25, no.~2, pp. 944--966, 2015.

\bibitem{nedic2017achieving}
A.~Nedic, A.~Olshevsky, and W.~Shi, ``Achieving geometric convergence for distributed optimization over time-varying graphs,'' \emph{SIAM Journal on Optimization}, vol.~27, no.~4, pp. 2597--2633, 2017.

\bibitem{liao2022compressed}
Y.~Liao, Z.~Li, K.~Huang, and S.~Pu, ``A compressed gradient tracking method for decentralized optimization with linear convergence,'' \emph{IEEE Transactions on Automatic Control}, vol.~67, no.~10, pp. 5622--5629, 2022.

\bibitem{li2021accelerated}
H.~Li and Z.~Lin, ``Accelerated gradient tracking over time-varying graphs for decentralized optimization,'' \emph{arXiv preprint arXiv:2104.02596}, 2021.

\bibitem{lu2020nesterov}
Q.~L{\"u}, X.~Liao, H.~Li, and T.~Huang, ``A {N}esterov-like gradient tracking algorithm for distributed optimization over directed networks,'' \emph{IEEE Transactions on Systems, Man, and Cybernetics: Systems}, 2020.

\bibitem{sun2022distributed}
Y.~Sun, G.~Scutari, and A.~Daneshmand, ``Distributed optimization based on gradient tracking revisited: Enhancing convergence rate via surrogation,'' \emph{SIAM Journal on Optimization}, vol.~32, no.~2, pp. 354--385, 2022.

\bibitem{mateos2016distributed}
D.~Mateos-N{\'u}nez and J.~Cort{\'e}s, ``Distributed saddle-point subgradient algorithms with laplacian averaging,'' \emph{IEEE Transactions on Automatic Control}, vol.~62, no.~6, pp. 2720--2735, 2016.

\bibitem{cortes2019distributed}
J.~Cort{\'e}s and S.~K. Niederl{\"a}nder, ``Distributed coordination for nonsmooth convex optimization via saddle-point dynamics,'' \emph{Journal of Nonlinear Science}, vol.~29, pp. 1247--1272, 2019.

\bibitem{aybat2016primal}
N.~S. Aybat and E.~Yazdandoost~Hamedani, ``A primal-dual method for conic constrained distributed optimization problems,'' \emph{Advances in neural information processing systems}, vol.~29, 2016.

\bibitem{mateos2010distributed}
G.~Mateos, J.~A. Bazerque, and G.~B. Giannakis, ``Distributed sparse linear regression,'' \emph{IEEE Transactions on Signal Processing}, vol.~58, no.~10, pp. 5262--5276, 2010.

\bibitem{mota2013d}
J.~F. Mota, J.~M. Xavier, P.~M. Aguiar, and M.~P{\"u}schel, ``D-admm: A communication-efficient distributed algorithm for separable optimization,'' \emph{IEEE Transactions on Signal processing}, vol.~61, no.~10, pp. 2718--2723, 2013.

\bibitem{chang2014multi}
T.-H. Chang, M.~Hong, and X.~Wang, ``Multi-agent distributed optimization via inexact consensus {ADMM},'' \emph{IEEE Transactions on Signal Processing}, vol.~63, no.~2, pp. 482--497, 2014.

\bibitem{shorinwa2020scalable}
O.~Shorinwa, T.~Halsted, and M.~Schwager, ``Scalable distributed optimization with separable variables in multi-agent networks,'' in \emph{2020 American Control Conference (ACC)}.\hskip 1em plus 0.5em minus 0.4em\relax IEEE, 2020, pp. 3619--3626.

\bibitem{carli2019distributed}
R.~Carli and M.~Dotoli, ``Distributed alternating direction method of multipliers for linearly constrained optimization over a network,'' \emph{IEEE Control Systems Letters}, vol.~4, no.~1, pp. 247--252, 2019.

\bibitem{zhang2018consensus}
Y.~Zhang and M.~M. Zavlanos, ``A consensus-based distributed augmented lagrangian method,'' in \emph{2018 IEEE Conference on Decision and Control (CDC)}.\hskip 1em plus 0.5em minus 0.4em\relax IEEE, 2018, pp. 1763--1768.

\bibitem{jakovetic2014linear}
D.~Jakoveti{\'c}, J.~M. Moura, and J.~Xavier, ``Linear convergence rate of a class of distributed augmented lagrangian algorithms,'' \emph{IEEE Transactions on Automatic Control}, vol.~60, no.~4, pp. 922--936, 2014.

\bibitem{kia2017augmented}
S.~S. Kia, ``An augmented lagrangian distributed algorithm for an in-network optimal resource allocation problem,'' in \emph{2017 American Control Conference (ACC)}.\hskip 1em plus 0.5em minus 0.4em\relax IEEE, 2017, pp. 3312--3317.

\bibitem{mokhtari2016network}
A.~Mokhtari, Q.~Ling, and A.~Ribeiro, ``Network newton distributed optimization methods,'' \emph{IEEE Transactions on Signal Processing}, vol.~65, no.~1, pp. 146--161, 2016.

\bibitem{liu2023communication}
H.~Liu, J.~Zhang, A.~M.-C. So, and Q.~Ling, ``A communication-efficient decentralized newton's method with provably faster convergence,'' \emph{IEEE Transactions on Signal and Information Processing over Networks}, 2023.

\bibitem{mansoori2019fast}
F.~Mansoori and E.~Wei, ``A fast distributed asynchronous newton-based optimization algorithm,'' \emph{IEEE Transactions on Automatic Control}, vol.~65, no.~7, pp. 2769--2784, 2019.

\bibitem{soori2020dave}
S.~Soori, K.~Mishchenko, A.~Mokhtari, M.~M. Dehnavi, and M.~Gurbuzbalaban, ``Dave-qn: A distributed averaged quasi-newton method with local superlinear convergence rate,'' in \emph{International Conference on Artificial Intelligence and Statistics}.\hskip 1em plus 0.5em minus 0.4em\relax PMLR, 2020, pp. 1965--1976.

\bibitem{mokhtari2016decentralized}
A.~Mokhtari, W.~Shi, Q.~Ling, and A.~Ribeiro, ``A decentralized second-order method with exact linear convergence rate for consensus optimization,'' \emph{IEEE Transactions on Signal and Information Processing over Networks}, vol.~2, no.~4, pp. 507--522, 2016.

\bibitem{li2022communication}
Y.~Li, P.~G. Voulgaris, and N.~M. Freris, ``A communication efficient quasi-newton method for large-scale distributed multi-agent optimization,'' in \emph{ICASSP 2022-2022 IEEE International Conference on Acoustics, Speech and Signal Processing (ICASSP)}.\hskip 1em plus 0.5em minus 0.4em\relax IEEE, 2022, pp. 4268--4272.

\bibitem{dennis1977quasi}
J.~E. Dennis, Jr and J.~J. Mor{\'e}, ``Quasi-newton methods, motivation and theory,'' \emph{SIAM review}, vol.~19, no.~1, pp. 46--89, 1977.

\bibitem{davidon1959variable}
W.~Davidon, ``Variable metric method for minimization, argonne natl,'' \emph{Labs., ANL-5990 Rev}, 1959.

\bibitem{byrd1996analysis}
R.~H. Byrd, H.~F. Khalfan, and R.~B. Schnabel, ``Analysis of a symmetric rank-one trust region method,'' \emph{SIAM Journal on Optimization}, vol.~6, no.~4, pp. 1025--1039, 1996.

\bibitem{zhu2010discrete}
M.~Zhu and S.~Mart{\'\i}nez, ``Discrete-time dynamic average consensus,'' \emph{Automatica}, vol.~46, no.~2, pp. 322--329, 2010.

\bibitem{deuflhard1979affine}
P.~Deuflhard and G.~Heindl, ``Affine invariant convergence theorems for newton’s method and extensions to related methods,'' \emph{SIAM Journal on Numerical Analysis}, vol.~16, no.~1, pp. 1--10, 1979.

\bibitem{fontecilla1987convergence}
R.~Fontecilla, T.~Steihaug, and R.~A. Tapia, ``A convergence theory for a class of quasi-newton methods for constrained optimization,'' \emph{SIAM Journal on Numerical Analysis}, vol.~24, no.~5, pp. 1133--1151, 1987.

\bibitem{tapia1978quasi}
R.~Tapia, ``Quasi-newton methods for equality constrained optimization: Equivalence of existing methods and a new implementation,'' in \emph{Nonlinear programming 3}.\hskip 1em plus 0.5em minus 0.4em\relax Elsevier, 1978, pp. 125--164.

\bibitem{xin2019distributed}
R.~Xin and U.~A. Khan, ``Distributed heavy-ball: A generalization and acceleration of first-order methods with gradient tracking,'' \emph{IEEE Transactions on Automatic Control}, vol.~65, no.~6, pp. 2627--2633, 2019.

\bibitem{xu2015augmented}
J.~Xu, S.~Zhu, Y.~C. Soh, and L.~Xie, ``Augmented distributed gradient methods for multi-agent optimization under uncoordinated constant stepsizes,'' in \emph{2015 54th IEEE Conference on Decision and Control (CDC)}.\hskip 1em plus 0.5em minus 0.4em\relax IEEE, 2015, pp. 2055--2060.

\bibitem{shorinwa2023distributed}
O.~Shorinwa and M.~Schwager, ``Distributed model predictive control via separable optimization in multi-agent networks,'' \emph{IEEE Transactions on Automatic Control}, 2023.

\bibitem{shorinwa2024distributed}
------, ``Distributed conjugate gradient method via conjugate direction tracking,'' in \emph{2024 American Control Conference (ACC)}.\hskip 1em plus 0.5em minus 0.4em\relax IEEE, 2024, pp. 2066--2073.

\end{thebibliography}

        \clearpage
        \begin{appendices}
	\section{Numerical Evaluations of DQN}
\label{sec:supp_simulations_uncons}
We present results examining the performance of DQN in well-conditioned and poorly-conditioned quadratic optimization problems.

\subsection{Quadratic Programming}
\label{sec:sim_QP}
Quadratic programming (QP) problems are prevalent in many domains, e.g., in design, estimation, and analysis. We consider separable quadratic programs of the form:
\begin{equation}
	\label{eq:dis_QP}
	\minimize{x \in \mathbb{R}^{n}} \sum_{i = 1}^{N } \left(\frac{1}{2} x^{\T} P_{i} x + q_{i}^{\T}x \right),
\end{equation}
where agent $i$ has access to only its local problem data ${P_{i} \in \mbb{R}^{m_{i} \times n}}$ and ${q_{i} \in \mbb{R}^{n}}$, which often arises when the problem data is collected locally by each agent. We consider the quadratic program in \eqref{eq:dis_QP} with $50$ agents (${N = 50}$) and ${n = 40}$. We generate the local data $P_{i}$ by randomly generating the data matrix ${A_{i} \in \mbb{R}^{m_{i} \times n_{i}}}$, with ${P_{i} = A_{i}^{\T}A_{i}}$ and the local data $q_{i}$ by randomly generating the data vector ${b_{i} \in \mbb{R}^{m_{i}}}$. We randomly generate $m_{i}$ from the uniform distribution over the interval ${[5, 30)}$. We note the resulting local objective function is non-strongly-convex. Further, the local objective function of each agent does not have a unique optimal solution, since the local data $P_{i}$ is not full-rank,~${\forall i \in \mcal{V}}$. We assess the performance of each method in terms of the computation time required for convergence to an RSE of $1\mathrm{e}^{-10}$, in addition to the total number of bits of information shared by each agent prior to convergence of the algorithm. Further, we set the maximum number of iterations of each method at $1000$. We randomly initialize the iterate of each agent, and use the same initial iterate for all methods. We utilize a closed-form solution for the optimization problems arising in the primal update procedure of C-ADMM, making C-ADMM more competitive with other distributed algorithms in terms of computation time by eliminating the need for a nested iterative optimization solver. 

\subsubsection{Well-Conditioned Optimization Problems}
\label{sec:sim_well_cond}
Generally, first-order distributed optimization methods, which utilize the local gradients of the objective function, exhibit good convergence rates in well-conditioned optimization problems. However, the performance of these methods degrade notably in poorly-conditioned optimization problems. We begin by examining the performance of our distributed algorithm in comparison to other distributed optimization algorithms in well-conditioned quadratic programs, on a randomly-generated connected communication graph, with connectivity ratio ${\kappa = 0.569}$. We randomly generate the problem instance with the condition number of the aggregate Hessian ${P = \sum_{i \in \mcal{V}} P_{i}}$ given by $2.269$.

In the D-Newton algorithm, we solve for the local Newton step using matrix factorization and back-substitution, which we found to be faster than utilizing a conjugate gradient method. In evaluating the D-Newton algorithm, which involves compressing the local approximate Hessian of each agent, we implemented the \emph{Top-$K$} and the \emph{Rank-$K$} variants. However, we note that, in many cases, the \emph{Top-$K$} variant failed to converge in the larger-scale problems we consider in this section, although it did converge in smaller-scale problems, unlike the \emph{Rank-$K$} variant, which converged in all problems for appropriate values of $K$. We note that \emph{Rank-$K$} compression represents the optimal low-rank approximation of the Hessian matrix, explaining our observations. Consequently, we utilize the \emph{Rank-$K$} variant of the D-Newton algorithm in all the evaluations discussed in the rest of this paper. We set the number of communication rounds for multi-step consensus in the D-Newton algorithm at $15$, based on the simulation results provided in \cite{liu2023communication}, with ${M = 1}$, ${\theta = 0.03}$, and ${K = 25}$. We selected a value of $K$ such that the algorithm converged within the maximum allowable number of iterations, noting that the algorithm failed to converge at small values of $K$. DQN requires an initial estimate of the inverse Hessian, which we initialize to ${3I_{n}}$. In initializing \mbox{ESOM-$K$}, we set ${\epsilon = 0.1}$ and ${K = 10}$, where $K$ denotes the number of nested communication rounds. Interchangeably, we refer to the \mbox{ESOM-$K$} algorithm as \mbox{ESOM}.

Table~\ref{tab:well_conditioned_QP} provides the computation time required by each agent, in seconds, and the associated cumulative number of bits of information shared by each agent to its neighbors, in Megabytes (MB), to compute the optimal solution of the quadratic program in \eqref{eq:dis_QP}. The best-performing stats are shown in bold font. We note that DIGing-ATC achieves the fastest computation time, while incurring the least communication cost, highlighting the strong performance of distributed first-order methods on well-conditioned optimization problems. \mbox{$ABm$-DS} achieves the second-best computation time, while DQN requires a slightly longer computation time for convergence to the optimal solution, although DQN achieves the second-best communication cost, incurring a lower communication cost compared to \mbox{C-ADMM}. The distributed second-order algorithms ESOM and D-Newton incur significantly greater computation time and communication cost, compared to the other algorithms.

\begin{table}[th]
	\centering
	\caption{The cumulative computation time (in seconds) and the cumulative size of messages (MB) exchanged per agent in the well-conditioned quadratic program, with ${\kappa = 0.569}$.}
	\label{tab:well_conditioned_QP}
	\begin{adjustbox}{width=\linewidth}
		{\begin{tabular}{l c c}
				\toprule
				Algorithm & Computation Time (secs.) & Messages Exchanged (MB)  \\
				\midrule
				$ABm$-DS \cite{xin2019distributed} & $1.164 \mathrm{e}^{-3}$ & $1.370 \mathrm{e}^{-1}$ \\
				C-ADMM \cite{mateos2010distributed} & $1.916 \mathrm{e}^{-3}$ & $8.896 \mathrm{e}^{-2}$ \\
				DIGing-ATC \cite{nedic2017achieving} & $\bm{1.787 \mathrm{e}^{-4}}$ & $\bm{3.136 \mathrm{e}^{-2}}$ \\
				D-Newton \emph{Rank-$K$} \cite{liu2023communication} & $2.361 \mathrm{e}^{-1}$ & $3.402 \mathrm{e}^{0}$ \\
				ESOM \cite{mokhtari2016decentralized} & $1.872 \mathrm{e}^{-2}$ & $6.874 \mathrm{e}^{-1}$ \\
				DQN (ours) & $1.957 \mathrm{e}^{-3}$ & $3.648 \mathrm{e}^{-2}$ \\
				\bottomrule
		\end{tabular}}
	\end{adjustbox}
\end{table}

In Figure~\ref{fig:well_conditioned_QP_all_alg}, we show the convergence error of each agent per iteration of each algorithm for this problem instance. As depicted in Figure~\ref{fig:well_conditioned_QP_all_alg}, by leveraging information on the curvature, distributed second-order optimization methods generally converge faster than other methods in terms of the number of iterations required for convergence. However, this benefit often comes at the expense of an increase in the cumulative computation time and communication cost required for convergence. Quasi-Newton methods, such as DQN, seek to strike a good tradeoff between the number of iterations required for convergence and the associated cumulative computation time and communication cost, while leveraging an estimate of the curvature of the objective function.

\begin{figure}[th]
	\centering
	\includegraphics[width=\linewidth]{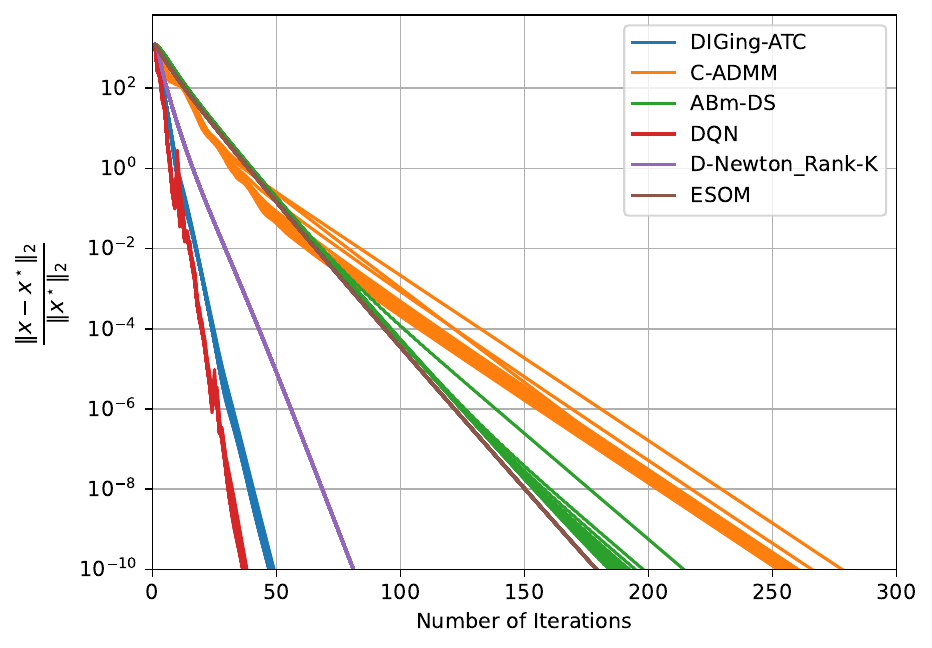}
	\caption{Convergence error of each agent per iteration in the distributed quadratic programming problem with a well-conditioned Hessian on a randomly-generated connected communication graph, with ${\kappa = 0.569}$.}
	\label{fig:well_conditioned_QP_all_alg}
\end{figure}

We proceed with examining the performance of each distributed algorithm on $20$ well-conditioned quadratic programs across a range of communication networks, with varying connectivity ratios, and provide the results in Table~\ref{tab:well_conditioned_QP_sweep}. We do not include the D-Newton \emph{Rank-$K$} algorithm in this evaluation, noting that D-Newton \emph{Rank-$K$} requires about an order of magnitude greater computation time and five-times greater communication cost than the distributed second-order method ESOM. We consider well-conditioned optimization problems, with the condition number of the problems ranging from $2.113$ to $2.568$. All methods converged to the optimal solution in all instances of the problem considered across all communication networks within the maximum number of iterations.
From Table~\ref{tab:well_conditioned_QP_sweep}, we note that DIGing-ATC achieves the fastest computation time across all the communication networks considered in this setup. In addition, DIGing-ATC incurs the least communication cost when ${\kappa = 0.57}$, consistent with the preceding discussion. C-ADMM requires communication of the least amount of information on sparsely-connected communication graphs, with ${\kappa = 0.21}$. DQN incurs the best performance in terms of the communication cost on more-densely-connected communication networks with ${\kappa = 0.74}$ and ${\kappa = 0.85}$. The second-order method ESOM incurs the greatest computation time and communication cost on these problems.

\begin{table*}[t]
	\centering
	\caption{The mean and standard deviation of the cumulative computation time (Comp. Time), in seconds (secs.), and the cumulative size of messages exchanged per agent (Comm. Cost), in Megabytes (MB), in $20$ well-conditioned quadratic programs, across a range of communication networks. All methods converged successfully to the specified tolerance on all problem instances.}
	\label{tab:well_conditioned_QP_sweep}
		{\begin{tabular}{l c c c c}
				\toprule
				\multirow{2}{*}{\textbf{Algorithm}} & \multicolumn{2}{c}{$\bm{\kappa} = \bm{0.21}$} & \multicolumn{2}{c}{$\bm{\kappa} = \bm{0.57}$} \\
				\cmidrule(lr){2-3} \cmidrule(lr){4-5}
				 & Comp. Time (secs.) & Comm. Cost (MB)
				 & Comp. Time (secs.) & Comm. Cost (MB) \\
				\midrule
				$ABm$-DS \cite{xin2019distributed} & $2.75\mathrm{e}^{-3}  \pm 7.16\mathrm{e}^{-4}$ & $1.98\mathrm{e}^{-1} \pm 5.24\mathrm{e}^{-2}$
				& $2.82\mathrm{e}^{-3}  \pm 7.23\mathrm{e}^{-4}$ & $1.94\mathrm{e}^{-1} \pm 4.43\mathrm{e}^{-2}$ \\
				C-ADMM \cite{mateos2010distributed} & $3.53\mathrm{e}^{-3}  \pm 7.67\mathrm{e}^{-4}$ & $\bm{1.24\mathrm{e}^{-1} \pm 2.14\mathrm{e}^{-2}}$
				& $3.21\mathrm{e}^{-3}  \pm 8.88\mathrm{e}^{-4}$ & $1.10\mathrm{e}^{-1} \pm 2.49\mathrm{e}^{-2}$ \\
				DIGing-ATC \cite{nedic2017achieving} & $\bm{1.96\mathrm{e}^{-3}  \pm 5.08\mathrm{e}^{-4}}$ & $2.08\mathrm{e}^{-1} \pm 5.01\mathrm{e}^{-2}$
				& $\bm{3.78\mathrm{e}^{-4}  \pm 7.26\mathrm{e}^{-5}}$ & $\bm{3.67\mathrm{e}^{-2} \pm 8.88\mathrm{e}^{-3}}$ \\
				ESOM \cite{mokhtari2016decentralized} & $6.42\mathrm{e}^{-3}  \pm 1.78\mathrm{e}^{-3}$ & $4.58\mathrm{e}^{-1} \pm 1.37\mathrm{e}^{-1}$
				& $6.97\mathrm{e}^{-3}  \pm 1.76\mathrm{e}^{-3}$ & $4.66\mathrm{e}^{-1} \pm 1.42\mathrm{e}^{-1}$ \\
				DQN (ours) & $1.82\mathrm{e}^{-2}  \pm 4.84\mathrm{e}^{-3}$ & $1.89\mathrm{e}^{-1} \pm 1.42\mathrm{e}^{-2}$
				& $3.89\mathrm{e}^{-3}  \pm 1.18\mathrm{e}^{-3}$ & $4.14\mathrm{e}^{-2} \pm 5.05\mathrm{e}^{-3}$ \\
				\bottomrule
		\end{tabular}}

	\bigskip

		{\begin{tabular}{l c c c c}
				\toprule
				\multirow{2}{*}{\textbf{Algorithm}} & \multicolumn{2}{c}{$\bm{\kappa} = \bm{0.74}$} & \multicolumn{2}{c}{$\bm{\kappa} = \bm{0.85}$} \\
				\cmidrule(lr){2-3} \cmidrule(lr){4-5}
		 & Comp. Time (secs.) & Comm. Cost (MB)
		 & Comp. Time (secs.) & Comm. Cost (MB) \\
				\midrule
				$ABm$-DS \cite{xin2019distributed} & $2.67\mathrm{e}^{-3}  \pm 6.14\mathrm{e}^{-4}$ & $1.89\mathrm{e}^{-1} \pm 3.86\mathrm{e}^{-2}$
				& $2.98\mathrm{e}^{-3}  \pm 8.62\mathrm{e}^{-4}$ & $1.96\mathrm{e}^{-1} \pm 4.21\mathrm{e}^{-2}$ \\
				C-ADMM \cite{mateos2010distributed} & $3.09\mathrm{e}^{-3}  \pm 1.09\mathrm{e}^{-3}$ & $1.04\mathrm{e}^{-1} \pm 2.81\mathrm{e}^{-2}$
				& $3.30\mathrm{e}^{-3}  \pm 1.10\mathrm{e}^{-3}$ & $1.04\mathrm{e}^{-1} \pm 2.96\mathrm{e}^{-2}$ \\
				DIGing-ATC \cite{nedic2017achieving} & $\bm{3.44\mathrm{e}^{-4}  \pm 1.03\mathrm{e}^{-4}}$ & $3.05\mathrm{e}^{-2} \pm 7.08\mathrm{e}^{-3}$
				& $\bm{3.08\mathrm{e}^{-4}  \pm 9.88\mathrm{e}^{-5}}$ & $2.62\mathrm{e}^{-2} \pm 6.03\mathrm{e}^{-3}$ \\
				ESOM \cite{mokhtari2016decentralized} & $6.31\mathrm{e}^{-3}  \pm 2.33\mathrm{e}^{-3}$ & $4.01\mathrm{e}^{-1} \pm 1.51\mathrm{e}^{-1}$
				& $6.10\mathrm{e}^{-3}  \pm 1.74\mathrm{e}^{-3}$ & $3.99\mathrm{e}^{-1} \pm 1.09\mathrm{e}^{-1}$ \\
				DQN (ours) & $2.97\mathrm{e}^{-3}  \pm 7.00\mathrm{e}^{-4}$ & $\bm{2.75\mathrm{e}^{-2} \pm 1.57\mathrm{e}^{-3}}$
				& $2.42\mathrm{e}^{-3}  \pm 5.49\mathrm{e}^{-4}$ & $\bm{2.27\mathrm{e}^{-2} \pm 8.74\mathrm{e}^{-4}}$ \\
				\bottomrule
		\end{tabular}}
\end{table*}

\subsubsection{Poorly-Conditioned Optimization Problems}
\label{sec:sim_poorly_cond}
We consider quadratic programs where the Hessian of the objective function is not well-conditioned. Generally, second-order methods exhibit good empirical convergence rates in these problems. In contrast, the convergence rates of first-order optimization methods typically degrade significantly in these problems. As in the well-conditioned setting, we assess the performance of each distributed algorithm on a randomly-generated connected communication network, with ${\kappa = 0.569}$, and a randomly-generated quadratic program, with the condition number of the aggregate Hessian given by $75.367$. We implement the D-Newton \emph{Rank-$K$} algorithm with ${K = 25}$, ${M = 1}$, and ${\theta = 0.03}$ and utilize $15$ communication rounds for the multi-step consensus procedure. Further, we initialize DQN with the estimate of the inverse Hessian of each agent's local objective function given by $0.1 I_{n}$. We set ${\epsilon = 0.1}$ and ${K = 10}$ in \mbox{ESOM-$K$}.

We present the cumulative computation time and the cumulative size of messages exchanged per agent in each distributed algorithm in Table~\ref{tab:poorly_conditioned_QP}. In addition, we indicate the success status of each algorithm in converging to the optimal solution. From Table~\ref{tab:poorly_conditioned_QP}, we note that the first-order algorithms \mbox{$ABm$-DS} and DIGing-ATC fail to converge to the optimal solution within the maximum number of iterations, with the exception of \mbox{C-ADMM}, which is generally less sensitive to the condition number of the problem, provided that a suitable value is selected for the penalty parameter. In contrast, DQN and the second-order algorithms D-Newton \emph{Rank-$K$} and ESOM converge to the optimal solution within $1000$ iterations. In this problem, C-ADMM requires the least computation time for convergence to the optimal solution, while DQN incurs the least communication cost for convergence.

\begin{table}[th]
	\centering
	\caption{The cumulative computation time (in seconds), the cumulative size of messages exchanged per agent, and the convergence status in the poorly-conditioned quadratic program, with ${\kappa = 0.569}$.}
	\label{tab:poorly_conditioned_QP}
	\begin{adjustbox}{width=\linewidth}
		{\begin{tabular}{l c c c}
				\toprule
				Algorithm & Computation Time (secs.) & Messages Exchanged (MB) & Converged \\
				\midrule
				$ABm$-DS \cite{xin2019distributed} & --- & --- & \xmark \\
				C-ADMM \cite{mateos2010distributed} & $\bm{5.239 \mathrm{e}^{-3}}$ & $2.592 \mathrm{e}^{-1}$ & \cmark \\
				DIGing-ATC \cite{nedic2017achieving} & --- & --- & \xmark \\
				D-Newton \emph{Rank-$K$} \cite{liu2023communication} & $4.224 \mathrm{e}^{-1}$ & $6.090 \mathrm{e}^{0}$ & \cmark \\
				ESOM \cite{mokhtari2016decentralized} & $2.439 \mathrm{e}^{-2}$ & $9.101 \mathrm{e}^{-1}$ & \cmark \\
				DQN (ours) & $7.144 \mathrm{e}^{-3}$ & $\bm{1.248 \mathrm{e}^{-1}}$ & \cmark \\
				\bottomrule
		\end{tabular}}
	\end{adjustbox}
\end{table}

In Figure~\ref{fig:poorly_conditioned_QP_all_alg}, we show the convergence error of each agent per iteration of each distributed algorithm for the poorly-conditioned quadratic program. As shown in Figure~\ref{fig:poorly_conditioned_QP_all_alg}, the performance of DIGing-ATC and $ABm$-DS degrades significantly, requiring more than $1000$ iterations for convergence, unlike C-ADMM. Although the second-order algorithms D-Newton \emph{Rank-$K$} and ESOM require significantly fewer iterations for convergence, these methods still require a greater computation time, along with a greater communication cost, compared to \mbox{C-ADMM}. In contrast, DQN provides a better tradeoff between the number of iterations required for convergence and the associated computation and communication cost.

\begin{figure}[th]
	\centering
	\includegraphics[width=\linewidth]{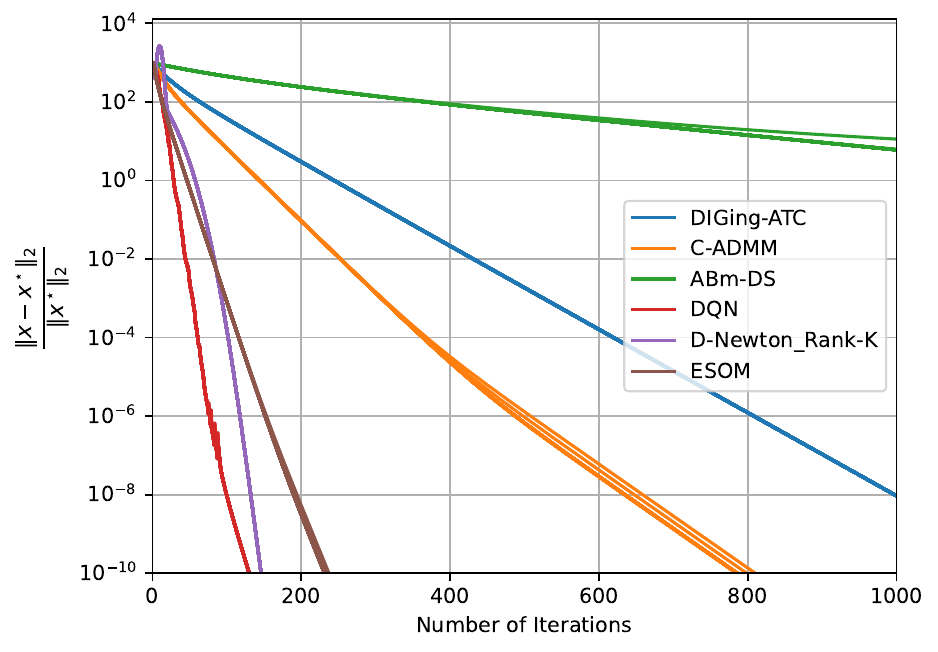}
	\caption{Convergence error of each agent per iteration in the distributed quadratic programming problem with a \emph{poorly-conditioned} Hessian on a randomly-generated connected communication graph, with ${\kappa = 0.569}$. The performance of DIGing-ATC and $ABm$-DS degrades notably in poorly-conditioned problems, compared to C-ADMM, DQN, and the second-order methods D-Newton \emph{Rank-$K$} and ESOM.}
	\label{fig:poorly_conditioned_QP_all_alg}
\end{figure}

Further, we examine the performance of the distributed optimization methods on $20$ randomly-generated poorly-conditioned quadratic programs on different communication graphs, with the condition number of the aggregate Hessian ranging between $42.339$ and $172.149$. We initialize the estimate of the inverse Hessian of each agent's local objective function to $0.1 I_{n}$ in DQN, and further initialize \mbox{ESOM-$K$} with ${\epsilon = 0.1}$ and ${K = 10}$. We do not include D-Newton \emph{Rank-$K$} in this evaluation, noting that D-Newton \emph{Rank-$K$} requires more than one order of magnitude greater computation time and about seven-times greater communication cost compared to the best-competing second-order method ESOM. In Table~\ref{tab:poorly_conditioned_QP_sweep_success}, we provide the mean and standard deviation of the cumulative computation time and the cumulative size of the messages exchanged per agent in each algorithm. We compute the summary statistics over successful runs of each algorithm and provide the associated success rate of each algorithm in Table~\ref{tab:poorly_conditioned_QP_sweep_success}. We display the best-performing stats in bold font, considering only algorithms that converge on all problems. We display entries associated with a success rate below $100\%$ in red. We note that only DQN and ESOM attain a perfect success rate, converging to the optimal solution in all problems. In line with the preceding discussion, the first-order methods $ABm$-DS, DIGing-ATC, and C-ADMM attain low success rates, especially on sparsely-connected communication graphs. As noted in the Table~\ref{tab:poorly_conditioned_QP_sweep_success}, DQN incurs the least communication cost across all values of the connectivity ratio $\kappa$, reducing the communication cost of ESOM by about seven-times. Moreover, DQN attains the fastest computation time across all values of $\kappa$, with the exception of ${\kappa = 0.21}$, where ESOM attains a faster computation time. Even without a limit on the maximum number of iterations available to each algorithm for convergence, DQN still requires the least communication overhead in comparison to all other algorithms, including $ABm$-DS, DIGing-ATC, and C-ADMM, across all values of ${\kappa}$, except ${\kappa = 0.21}$.

\begin{table*}[t]
	\centering
	\caption{The mean and standard deviation of the cumulative computation time (Comp. Time), in seconds (secs.), and the cumulative size of messages exchanged per agent (Comm. Cost), in Megabytes (MB), in $20$ poorly-conditioned quadratic programs, across a range of communication networks, in addition to the success rate of each algorithm in converging to the optimal solution. Entries associated with a non-zero success rate below $100\%$ are displayed in red.}
	\label{tab:poorly_conditioned_QP_sweep_success}
	\begin{adjustbox}{width=\linewidth}
		{\begin{tabular}{l c c c c c c}
				\toprule
				\multirow{2}{*}{\textbf{Algorithm}} & \multicolumn{3}{c}{$\bm{\kappa} = \bm{0.21}$} & \multicolumn{3}{c}{$\bm{\kappa} = \bm{0.57}$} \\
				\cmidrule(lr){2-4} \cmidrule(lr){5-7}
				 & Comp. Time (secs.) & Comm. Cost (MB) & Success Rate ($\%$)
				 & Comp. Time (secs.) & Comm. Cost (MB) & Success Rate ($\%$) \\
				\midrule
				$ABm$-DS \cite{xin2019distributed} & --- & --- & $0$
				& --- & --- & $0$ \\
				C-ADMM \cite{mateos2010distributed} & {\color{red} $7.04\mathrm{e}^{-3}  \pm 1.37\mathrm{e}^{-3}$} & {\color{red} $2.49\mathrm{e}^{-1} \pm 4.23\mathrm{e}^{-2}$} & {\color{red} $45$}
				& {\color{red} $7.24\mathrm{e}^{-3}  \pm 1.40\mathrm{e}^{-3}$} & {\color{red} $2.51\mathrm{e}^{-1} \pm 4.27\mathrm{e}^{-2}$} & {\color{red} $60$} \\
				DIGing-ATC \cite{nedic2017achieving} & {\color{red} $5.90\mathrm{e}^{-3}$} & {\color{red} $5.65\mathrm{e}^{-1}$} & {\color{red} $5$}
				& {\color{red} $4.62\mathrm{e}^{-3}  \pm 7.63\mathrm{e}^{-4}$} & {\color{red} $5.10\mathrm{e}^{-1} \pm 8.99\mathrm{e}^{-2}$} & {\color{red} $15$} \\
				ESOM \cite{mokhtari2016decentralized} & $\bm{3.23\mathrm{e}^{-2}  \pm 1.07\mathrm{e}^{-2}}$ & $2.00\mathrm{e}^{0} \pm 3.36\mathrm{e}^{-1}$ & $100$
				& $1.74\mathrm{e}^{-2}  \pm 5.14\mathrm{e}^{-3}$ & $1.10\mathrm{e}^{0} \pm 2.55\mathrm{e}^{-1}$ & $100$ \\
				DQN (ours) & $3.85\mathrm{e}^{-2}  \pm 1.15\mathrm{e}^{-2}$ & $\bm{4.10\mathrm{e}^{-1} \pm 1.19\mathrm{e}^{-1}}$ & $100$
				& $\bm{1.63\mathrm{e}^{-2}  \pm 4.14\mathrm{e}^{-3}}$ & $\bm{1.61\mathrm{e}^{-1} \pm 2.45\mathrm{e}^{-2}}$ & $100$ \\
				\bottomrule
		\end{tabular}}
	\end{adjustbox}

	\bigskip

	\begin{adjustbox} {width=\linewidth}
		{\begin{tabular}{l c c c c c c}
				\toprule
				\multirow{2}{*}{\textbf{Algorithm}} & \multicolumn{3}{c}{$\bm{\kappa} = \bm{0.74}$} & \multicolumn{3}{c}{$\bm{\kappa} = \bm{0.85}$} \\
				\cmidrule(lr){2-4} \cmidrule(lr){5-7}
		 & Comp. Time (secs.) & Comm. Cost (MB) & Success Rate ($\%$)
		 & Comp. Time (secs.) & Comm. Cost (MB) & Success Rate ($\%$) \\
				\midrule
				$ABm$-DS \cite{xin2019distributed} & --- & --- & $0$
				& --- & --- & $0$ \\
				C-ADMM \cite{mateos2010distributed} & {\color{red} $7.48\mathrm{e}^{-3}  \pm 1.53\mathrm{e}^{-3}$} & {\color{red} $2.58\mathrm{e}^{-1} \pm 4.38\mathrm{e}^{-2}$} & {\color{red} $70$}
				& {\color{red} $8.29\mathrm{e}^{-3} \pm 2.02\mathrm{e}^{-3}$} & {\color{red} $2.58\mathrm{e}^{-1} \pm 4.58\mathrm{e}^{-2}$} & {\color{red} $70$} \\
				DIGing-ATC \cite{nedic2017achieving} & {\color{red} $4.84\mathrm{e}^{-3}  \pm 9.70\mathrm{e}^{-4}$} & {\color{red} $5.12\mathrm{e}^{-1} \pm 8.84\mathrm{e}^{-2}$} & {\color{red} $20$}
				& {\color{red} $5.27\mathrm{e}^{-3}  \pm 9.30\mathrm{e}^{-4}$} & {\color{red} $4.97\mathrm{e}^{-1} \pm 8.71\mathrm{e}^{-2}$} & {\color{red} $20$} \\
				ESOM \cite{mokhtari2016decentralized}& $1.78\mathrm{e}^{-2}  \pm 5.06\mathrm{e}^{-3}$ & $1.08\mathrm{e}^{0} \pm 2.59\mathrm{e}^{-1}$ & $100$
				& $1.94\mathrm{e}^{-2}  \pm 8.21\mathrm{e}^{-3}$ & $1.08\mathrm{e}^{0} \pm 2.34\mathrm{e}^{-1}$ & $100$ \\
				DQN (ours) & $\bm{1.37\mathrm{e}^{-2}  \pm 3.07\mathrm{e}^{-3}}$ & $\bm{1.52\mathrm{e}^{-1} \pm 4.01\mathrm{e}^{-2}}$ & $100$
				& $\bm{1.27\mathrm{e}^{-2}  \pm 4.75\mathrm{e}^{-3}}$ & $\bm{1.48\mathrm{e}^{-1} \pm 3.26\mathrm{e}^{-2}}$ & $100$ \\
				\bottomrule
		\end{tabular}}
	\end{adjustbox}
\end{table*}

	\section{Numerical Evaluations of EC-DQN}
\label{sec:supp_simulations_cons}
We present evaluation results of EC-DQN in well-conditioned basis pursuit problems, in comparison to existing algorithms for distributed constrained optimization. The problem setup is described in the paper.

\subsection{Basis Pursuit Denoising: Well-Conditioned Problems}
We examine the performance of each algorithm in $20$ \emph{well-conditioned} optimization problems on randomly-generated communication networks, with ${\kappa}$ varying from $0.219$ to $0.856$ and condition number of the Hessian ranging between $2.578$ and $2.097$.
We present the computation time and associated communication cost required for convergence per agent in each algorithm in Table~\ref{tab:well_conditioned_basis_run}. When ${\kappa = 0.22}$, SONATA achieves the fastest computation time, while \mbox{C-ADMM} requires the least communication cost for convergence to the optimal solution. However, for all other values of $\kappa$, EC-DQN attains the fastest computation time, while incurring the least communication cost for convergence. In these problems, C-ADMM requires a notably greater computation time relative to the other algorithms. DPDA fails to converge to the desired convergence threshold within the maximum number of iterations. In Figure~\ref{fig:well_conditioned_basis_run}, we examine the performance of each algorithm on a problem with condition number $2.312$ and a communication graph with ${\kappa = 0.578}$. We note that EC-DQN requires the fewest number of iterations for convergence to the optimal solution. Further, EC-DQN attains the fastest computation time and least communication cost of $3.17\mathrm{e}^{-2}$~secs. and $8.81\mathrm{e}^{-3}$~MB, respectively, compared to the best-competing algorithm SONATA with a computation time and communication cost of $8.52\mathrm{e}^{-2}$~secs. and $2.32\mathrm{e}^{-2}$~MB, respectively. Although the convergence error of the agents' iterates in DPDA decreases at each iteration, it does not decrease fast enough to attain convergence within $1000$ iterations. C-ADMM, however, converges in about $400$ iterations.

\begin{table*}[t]
	\centering
	\caption{The mean and standard deviation of the cumulative computation time (Comp. Time), in seconds (secs.), and communication cost (Comm. Cost), in Megabytes (MB), per agent, in $20$ well-conditioned basis pursuit denoising problems, on different communication networks.}
	\label{tab:well_conditioned_basis_run}
		{\begin{tabular}{l c c c c}
				\toprule
				\multirow{2}{*}{\textbf{Algorithm}} & \multicolumn{2}{c}{$\bm{\kappa} = \bm{0.22}$} & \multicolumn{2}{c}{$\bm{\kappa} = \bm{0.57}$} \\
				\cmidrule(lr){2-3} \cmidrule(lr){4-5}
				 & Comp. Time (secs.) & Comm. Cost (MB)
				 & Comp. Time (secs.) & Comm. Cost (MB) \\
				\midrule
				C-ADMM \cite{mateos2010distributed} & $2.73\mathrm{e}^{-2}  \pm 1.32\mathrm{e}^{-2}$ & $\bm{7.93\mathrm{e}^{-2} \pm 4.23\mathrm{e}^{-2}}$
				& $3.44\mathrm{e}^{-2}  \pm 1.91\mathrm{e}^{-2}$ & $8.93\mathrm{e}^{-2} \pm 4.36\mathrm{e}^{-2}$ \\
				DPDA \cite{aybat2016primal} & --- & ---
				& --- & --- \\
				SONATA \cite{sun2022distributed} & $\bm{1.74\mathrm{e}^{-2}  \pm 4.62\mathrm{e}^{-3}}$ & $9.80\mathrm{e}^{-2} \pm 3.94\mathrm{e}^{-2}$
				& $6.72\mathrm{e}^{-3}  \pm 2.13\mathrm{e}^{-3}$ & $4.61\mathrm{e}^{-2} \pm 1.80\mathrm{e}^{-2}$ \\
				EC-DQN (ours) & $1.91\mathrm{e}^{-2}  \pm 5.60\mathrm{e}^{-3}$ & $1.28\mathrm{e}^{-1} \pm 1.72\mathrm{e}^{-2}$
				& $\bm{4.95\mathrm{e}^{-3}  \pm 1.55\mathrm{e}^{-3}}$ & $\bm{3.24\mathrm{e}^{-2} \pm 3.32\mathrm{e}^{-3}}$ \\
				\bottomrule
		\end{tabular}}

	\bigskip

		{\begin{tabular}{l c c c c}
				\toprule
				\multirow{2}{*}{\textbf{Algorithm}} & \multicolumn{2}{c}{$\bm{\kappa} = \bm{0.75}$} & \multicolumn{2}{c}{$\bm{\kappa} = \bm{0.86}$} \\
				\cmidrule(lr){2-3} \cmidrule(lr){4-5}
		 & Comp. Time (secs.) & Comm. Cost (MB)
		 & Comp. Time (secs.) & Comm. Cost (MB) \\
				\midrule
				C-ADMM \cite{mateos2010distributed} & $2.83\mathrm{e}^{-2}  \pm 1.39\mathrm{e}^{-2}$ & $7.67\mathrm{e}^{-2} \pm 3.65\mathrm{e}^{-2}$
				& $2.49\mathrm{e}^{-2}  \pm 1.41\mathrm{e}^{-2}$ & $6.96\mathrm{e}^{-2} \pm3.54\mathrm{e}^{-2}$ \\
				DPDA \cite{aybat2016primal} & --- & ---
				& --- & --- \\
				SONATA \cite{sun2022distributed} & $5.29\mathrm{e}^{-3}  \pm 2.30\mathrm{e}^{-3}$ & $3.20\mathrm{e}^{-2} \pm 1.10\mathrm{e}^{-2}$
				& $7.08\mathrm{e}^{-3}  \pm 3.87\mathrm{e}^{-3}$ & $4.23\mathrm{e}^{-2} \pm 2.12\mathrm{e}^{-2}$ \\
				EC-DQN (ours) & $\bm{3.71\mathrm{e}^{-3}  \pm 1.08\mathrm{e}^{-3}}$ & $\bm{2.27\mathrm{e}^{-2} \pm 2.61\mathrm{e}^{-3}}$
				& $\bm{3.39\mathrm{e}^{-3}  \pm 8.60\mathrm{e}^{-4}}$ & $\bm{2.40\mathrm{e}^{-2} \pm 3.87\mathrm{e}^{-3}}$ \\
				\bottomrule
		\end{tabular}}
\end{table*}

\begin{figure}[th]
	\centering
	\includegraphics[width=\linewidth]{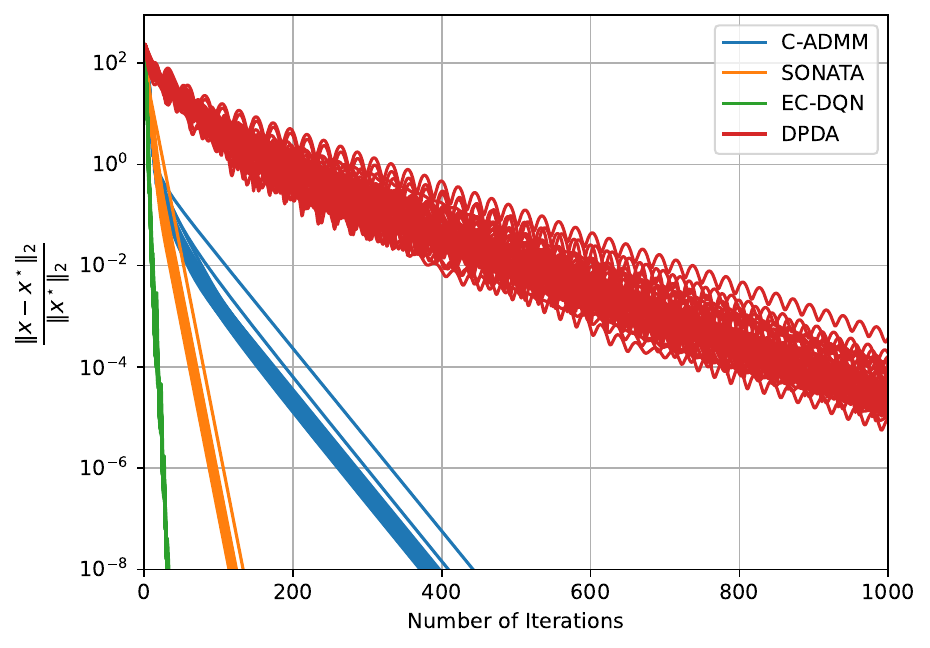}
	\caption{Convergence error of each algorithm on a randomly-generated communication graph with ${\kappa = 0.566}$ in a well-conditioned basis pursuit denoising problem. EC-DQN converges faster than all methods.}
	\label{fig:well_conditioned_basis_run}
\end{figure}
	\appendices
\label{sec:appendix}

\let\oldsubsection = \thesubsection
\renewcommand{\thesubsection}{\Alph{subsection}}

\section{Proof of Lemma 
1
}
\label{appdx:convergence_of_the_mean_sequence}
We prove convergence of the sequence ${\{\mean{\bm{x}}^{(k)}\}_{\forall k \geq 0}}$ to a limit point. 
Some of the ideas used in this proof can be found in \cite{xu2015augmented, shorinwa2023distributed}. From 
(14), (15), (16),
we obtain the following sequences:
\begin{align}
	\mean{\bm{x}}^{(k + 1)} &= \mean{\bm{x}}^{(k)} + \mean{\bm{\alpha}^{(k)} \bm{z}^{(k)}} \label{eq:dqn_mean_x_update} \\
	\mean{\bm{v}}^{(k + 1)} &= \mean{\bm{v}}^{(k)} + \mean{\bm{g}}^{(k + 1)} - \mean{\bm{g}}^{(k)} \label{eq:dqn_mean_v_update} \\
	\mean{\bm{z}}^{(k + 1)} &=\frac{\bm{1}_{N} \bm{1}_{N}^{\T}}{N} \bm{d}^{(k + 1)} \label{eq:dqn_mean_z_update}.
\end{align}

The following lemma bounds the disagreement error between the iterates of each agent and the mean of all the agents' iterates.

\stepcounter{lemma}
\begin{lemma}
	\label{lem:sequence_bounds}
    The sequences $\{\tilde{\bm{x}}^{(k)}\}_{\forall k \geq 0}$, $\{\tilde{\bm{v}}^{(k)}\}_{\forall k \geq 0}$, and $\{\tilde{\bm{z}}^{(k)}\}_{\forall k \geq 0}$, representing the disagreement between the agents on their corresponding local variables satisfy the following bounds:
    \begin{align}
        \norm{\tilde{\bm{x}}^{(k + 1)}}_{2} &\leq \lambda \norm{\tilde{\bm{x}}^{(k)}}_{2}
        + \lambda r_{\alpha} \norm{\mean{\bm{\alpha}^{(k)} \bm{z}^{(k)}}}_{2} \notag \\
        & \quad + \lambda^{2} \alpha_{\max} (1 + r_{\alpha}) q \left(\norm{\tilde{\bm{v}}^{(k)}}_{2} + \norm{\mean{\bm{g}}^{(k)}}_{2}\right), \label{eq:lemma_disagreement_x}  \\
        \norm{\tilde{\bm{v}}^{(k + 1)}}_{2} &\leq \left(\lambda +  \lambda^{3} L \alpha_{\max} (1 + r_{\alpha}) q  \right) \norm{\tilde{\bm{v}}^{(k)}}_{2} \notag \\
        & \quad + \lambda L ( 1 + \lambda) \norm{\tilde{\bm{x}}^{(k)}}_{2} \notag \\
        & \quad + \lambda^{3} L \alpha_{\max} (1 + r_{\alpha}) q \norm{\mean{\bm{g}}^{(k)}}_{2} \notag \\
        & \quad + \lambda L (1 + \lambda r_{\alpha}) \norm{\mean{\bm{\alpha}^{(k)} \bm{z}^{(k)}}}_{2}, \label{eq:lemma_disagreement_v} \\
        \norm{\tilde{\bm{z}}^{(k + 1)}}_{2} &\leq \lambda q \left(\norm{\tilde{\bm{v}}^{(k + 1)}}_{2} + \norm{\mean{\bm{g}}^{(k + 1)}}_{2}\right), \label{eq:lemma_disagreement_z}
    \end{align}
    respectively, given the recurrence in 
    (17), (11), (18),
    where ${q^{2} = \gamma^{2} \min\{n, N\}}$.
\end{lemma}

\begin{proof}
     We begin by expressing the disagreement between each agent's iterate and the mean iterate, given by:
     \begin{equation}
        \label{eq:disagreement_x}
         \begin{aligned}
             \tilde{\bm{x}}^{(k + 1)} &= {\bm{x}}^{(k + 1)} - \mean{\bm{x}}^{(k + 1)} \\
             &= W\left(\bm{x}^{(k)} + \bm{\alpha}^{(k)} \bm{z}^{(k)}\right) - \left(\mean{\bm{x}}^{(k)} + \mean{\bm{\alpha}^{(k)} \bm{z}^{(k)}} \right) \\
             &= M \tilde{\bm{x}}^{(k)} + M \left(\bm{\alpha}^{(k)} \bm{z}^{(k)} - \mean{\bm{\alpha}^{(k)} \bm{z}^{(k)}} \right).
         \end{aligned}
     \end{equation}
     Similarly, we have that:
     \begin{equation}
        \label{eq:disagreement_v}
         \begin{aligned}
             \tilde{\bm{v}}^{(k + 1)} &= {\bm{v}}^{(k + 1)} - \mean{\bm{v}}^{(k + 1)} \\
             &= M \tilde{\bm{v}}^{(k)} + M \left(\nabla \bm{f}\bp{\bm{x}^{(k + 1)}} - \nabla \bm{f}\bp{\bm{x}^{(k)}} \right),
         \end{aligned}
     \end{equation}
     with ${\tilde{\bm{z}}^{(k + 1)} = M \bm{d}^{(k + 1)}}$. Taking the norm of ${\tilde{\bm{z}}^{(k + 1)}}$ yields:
     \begin{equation}
        \label{eq:disagreement_z}
         \begin{aligned}
             \norm{\tilde{\bm{z}}^{(k + 1)}}_{2} &\leq \lambda \norm{\bm{d}^{(k + 1)}}_{2} \\
             &\leq \lambda q \norm{\bm{v}^{(k + 1)}}_{2} \\
             &\leq \lambda q \norm{\tilde{\bm{v}}^{(k + 1)} + \mean{\bm{v}}^{(k + 1)}}_{2} \\
             &\leq \lambda q \left(\norm{\tilde{\bm{v}}^{(k + 1)}}_{2} + \norm{\mean{\bm{g}}^{(k + 1)}}_{2}\right),
         \end{aligned}
     \end{equation}
     where we use the relation: ${\mean{\bm{v}}^{(k + 1)} = \mean{\bm{g}}^{(k + 1)}}$. Likewise, from \eqref{eq:disagreement_x} and \eqref{eq:disagreement_v}, we obtain:
     \begin{equation}
        \label{eq:disagreement_x_mid_step}
         \begin{aligned}
             \norm{\tilde{\bm{x}}^{(k + 1)}}_{2} &\leq \lambda \norm{\tilde{\bm{x}}^{(k)}}_{2} + \lambda \alpha_{\max} (1 + r_{\alpha}) \norm{\tilde{\bm{z}}^{(k)}}_{2} \\
             & \quad + \lambda r_{\alpha} \norm{\mean{\bm{\alpha}^{(k)} \bm{z}^{(k)}}}_{2},
         \end{aligned}
     \end{equation}
     \begin{equation}
        \label{eq:disagreement_v_mid_step}
         \begin{aligned}
             \norm{\tilde{\bm{v}}^{(k + 1)}}_{2} &\leq \lambda \norm{\tilde{\bm{v}}^{(k)}}_{2} + \lambda \norm{\nabla \bm{f}\bp{\bm{x}^{(k + 1)}} - \nabla \bm{f}\bp{\bm{x}^{(k)}}}_{2} \\
             &\leq \lambda L \left(\norm{\tilde{\bm{x}}^{(k + 1)}}_{2} + \norm{\tilde{\bm{x}}^{(k)}}_{2} + \norm{\mean{\bm{\alpha}^{(k)} \bm{z}^{(k)}}}_{2} \right) \\
             & \quad + \lambda \norm{\tilde{\bm{v}}^{(k)}}_{2}.
         \end{aligned}
     \end{equation}
     Applying \eqref{eq:disagreement_z} in \eqref{eq:disagreement_x_mid_step} and \eqref{eq:disagreement_v_mid_step} yields \eqref{eq:lemma_disagreement_x} and \eqref{eq:lemma_disagreement_v}, respectively.

\end{proof}

Upon defining the following sequences:
\begin{align}
	X^{(k)} &= \sqrt{\sum_{l = 0}^{k} \norm{\tilde{\bm{x}}^{(l)}}_{2}^{2}}, \quad
	V^{(k)} = \sqrt{\sum_{l = 0}^{k} \norm{\tilde{\bm{v}}^{(l)}}_{2}^{2}}, \label{eq:sqrt_X_V_definition} \\
	R^{(k)} &= \sqrt{\sum_{l = 0}^{k} \left(\norm{\mean{\bm{\alpha}^{(l)} \bm{z}^{(l)}}}_{2}^{2}
	+ \norm{\mean{\bm{g}}^{(l + 1)}}_{2}^{2} \right)
	}, \label{eq:sqrt_R_definition}
\end{align}
we obtain the following bounds on $X^{(k)}$ and $V^{(k)}$ from Lemma~\ref{lem:sequence_bounds}:
\begin{align}
	X^{(k)} \leq \mu_{xr}R^{(k)} + \mu_{xc}, \label{eq:X_bound} \\
	V^{(k)} \leq \theta_{vr}R^{(k)} + \theta_{vc}, \label{eq:V_bound}
\end{align}
where ${\mu_{xr} = \frac{\rho_{xr} + \rho_{xv} \rho_{vr}}{\eta}}$, ${\mu_{xc} = \frac{\rho_{xv} \epsilon_{v} + \epsilon}{\eta}}$,
${\theta_{vr} = \frac{\rho_{vx}\rho_{xr} + \rho_{vr}}{\eta}}$, ${\theta_{vc} = \frac{\rho_{vx} \epsilon{x} + \epsilon_{v}}{\eta}}$,
${\eta = 1 - \rho_{vx}\rho_{xv}}$,
with
${\rho_{xv} = \frac{\sqrt{2}}{1 - \lambda} \lambda^{2} \alpha_{\max} (1 + r_{\alpha}) q}$,
${\rho_{vx} = \frac{\sqrt{2}}{1 - \lambda} \lambda L (1 + \lambda)}$,
${\rho_{xr} = \frac{\sqrt{2}}{1 - \lambda} \left(\lambda^{2} \alpha_{\max} (1 + r_{\alpha}) q + \lambda r_{\alpha} \right)}$,
${\rho_{vr} = \frac{\sqrt{2}}{1 - \lambda} \left( \lambda^{3} L \alpha_{\max} (1 + r_{\alpha}) q + \lambda L (1 + \lambda r_{\alpha}) \right)}$,
${\epsilon_{x} = \norm{\tilde{\bm{x}}^{(0)}}_{2} \sqrt{\frac{2}{1 - \lambda^{2}}}}$, and
${\epsilon_{v} = \norm{\tilde{\bm{v}}^{(0)}}_{2} \sqrt{\frac{2}{1 - \lambda^{2}}}}$.

Provided that ${\alpha_{\max} < \frac{1 - \lambda}{\lambda^{3} L (1 + r_{\alpha}) q}}$, the bounds in \eqref{eq:X_bound} and \eqref{eq:V_bound} arise from direct application of \cite[cf. Lemma~2]{xu2015augmented}, which applies from non-negativity of the $\norm{\cdot}_{2}$, with ${w(k)}$ in \cite[cf. Lemma~2]{xu2015augmented} defined accordingly from the right-hand-side of \eqref{eq:lemma_disagreement_x} and \eqref{eq:lemma_disagreement_v}.

Further, from $L$-Lipschitz continuity of the gradient of $f_{i}$,~${\forall i \in \mcal{V}}$:
\begin{equation}
	\begin{aligned}
		\bm{f}(\mean{\bm{x}}^{(k + 1)}) &\leq \bm{f}(\mean{\bm{x}}^{(k)}) + \trace{g(\mean{\bm{x}}^{(k)})(\mean{\bm{x}}^{(k + 1)} - \mean{\bm{x}}^{(k)})^{\T}} \\
            & +\frac{L}{2} \norm{\mean{\bm{x}}^{(k + 1)} - \mean{\bm{x}}^{(k)}}_{2}^{2} \\
            &\leq \bm{f}(\mean{\bm{x}}^{(k)}) + \trace{g(\bm{x}^{(k)})\left(\mean{\bm{\alpha}^{(k)} \bm{z}^{(k)}}\right)^{\T}} \\
            & +\frac{L}{2} \norm{\mean{\bm{\alpha}^{(k)} \bm{z}^{(k)}}}_{2}^{2} \\
            &+ \trace{\left(g(\mean{\bm{x}}^{(k)}) - g(\bm{x}^{(k)})\right)\left(\mean{\bm{\alpha}^{(k)} \bm{z}^{(k)}}\right)^{\T}},
	\end{aligned}
\end{equation}
which simplifies to:
\begin{equation}
	\begin{aligned}
		\bm{f}(\mean{\bm{x}}^{(t + 1)}) &\leq \bm{f}(\mean{\bm{x}}^{(0)}) - \frac{1 - 2 \xi}{2} \left(R^{(t)} \right)^{2} \\
            &
		+ L X^{(t)} R^{(t)} 
            + \theta_{\alpha} V^{(t)} R^{(t)},
	\end{aligned}
\end{equation}
upon summing over $k$ from $0$ to $t$, using \eqref{eq:dqn_mean_x_update}, \eqref{eq:dqn_mean_v_update}, \eqref{eq:dqn_mean_z_update}, \eqref{eq:sqrt_X_V_definition} and \eqref{eq:sqrt_R_definition}, where ${\xi = (\frac{L}{2} + N^{2}) \alpha_{\max}^{2}q^{2}}$ and ${\theta_{\alpha} = \beta N^{2} + c (\frac{L}{2} + N^{2})\alpha_{\max}^{2}q^{2}}$, with ${\beta, c > 0}$.
Using \eqref{eq:X_bound} and \eqref{eq:V_bound} for $X^{(t)}$ and $V^{(t)}$, we obtain:
\begin{equation}
	\bm{f}(\mean{\bm{x}}^{(t + 1)}) \leq \bm{f}(\mean{\bm{x}}^{(0)}) + a \left(R^{(t)} \right)^{2}
	+ b R^{(t)},
\end{equation}
where ${a = - \frac{1 - 2 \xi}{2} + L \mu_{xr} + \theta_{\alpha}\theta_{vr}}$ and ${b = L \mu_{xc} + \theta_{\alpha}\theta_{vc}}$.
To bound the error in the objective value with respect to the optimal objective value, we subtract $f^{\star}$ from both sides, which results in:
\begin{equation}
    \label{eq:obj_value_error_step_a}
    \bm{f}(\mean{\bm{x}}^{(t + 1)}) - f^{\star} \leq \bm{f}(\mean{\bm{x}}^{(0)}) - f^{\star} + a \left(R^{(t)} \right)^{2}
	+ b R^{(t)}.
\end{equation}
Since the left-hand side of \eqref {eq:obj_value_error_step_a} is lower-bounded by $0$, we obtain:
\begin{equation}
    \label{eq:obj_value_error_step_b}
    \bm{f}(\mean{\bm{x}}^{(0)}) - f^{\star} + a \left(R^{(t)} \right)^{2}
	+ b R^{(t)} \geq 0,
\end{equation}
which is a quadratic function of $R^{(t)}$.
Provided that ${a < 0}$, the following relation holds from the properties of quadratic functions:
\begin{equation}
	\label{eq:R_bound_lim}
	\lim_{t \rightarrow \infty} R^{(t)} \leq \frac{b + \sqrt{b^{2} - 4 a \left(\bm{f}(\mean{\bm{x}}^{(0)}) - f^{\star}\right)}}{- 2 a} = U < \infty.
\end{equation}
Applying the monotone convergence theorem, we note that ${R^{(t)} \rightarrow U}$ as ${t \rightarrow \infty}$, showing that:
\begin{equation}
	\lim_{k \rightarrow \infty} \left(\norm{\mean{\bm{g}}^{(k)}}_{2}^{2} + \norm{\mean{\bm{\alpha}^{(k)} \bm{z}^{(k)}}}_{2}^{2} \right) = 0,
\end{equation}
which implies that:
\begin{equation}
    \label{eq:mean_g_convergence}
    \lim_{k \rightarrow \infty} \norm{\mean{\bm{g}}^{(k)}}_{2} = \lim_{k \rightarrow \infty}  \norm{\mean{\bm{\alpha}^{(k)} \bm{z}^{(k)}}}_{2} =  0.
\end{equation}
Further, from \eqref{eq:dqn_mean_x_update} and \eqref{eq:mean_g_convergence}:
\begin{equation}
	\lim_{k \rightarrow \infty} \norm{\mean{\bm{x}}^{(k + 1)} - \mean{\bm{x}}^{(k)}}_{2} = 0.
\end{equation}

\section{Proof of Theorem 
1
}
\label{appdx:convergence_to_the_mean}
Taking the limit of both sides in \eqref{eq:X_bound}, we obtain:
\begin{equation}
	\lim_{k \rightarrow \infty} X^{(k)} \leq \lim_{k \rightarrow \infty} \mu_{xr}R^{(k)} + \mu_{xc} 
    \leq \mu_{xr}U + \mu_{xc}
    < \infty,
\end{equation}
which results in the following relation:
\begin{equation}
    \label{eq:x_tilde_convergence_proof}
    \lim_{k \rightarrow \infty} \norm{\tilde{\bm{x}}^{(k)}}_{2} = 0,
\end{equation}
from \eqref{eq:sqrt_X_V_definition}, where we have applied the monotone convergence theorem. Likewise, from \eqref{eq:V_bound}, we have that:
\begin{equation}
    \label{eq:v_tilde_convergence_proof}
    \lim_{k \rightarrow \infty} V^{(k)} \leq \lim_{k \rightarrow \infty} \theta_{vr}R^{(k)} + \theta_{vc} 
    \leq \theta_{vr}U + \theta_{vc}
    < \infty,
\end{equation}
yielding: 
\begin{equation}
    \lim_{k \rightarrow \infty} \norm{\tilde{\bm{v}}^{(k)}}_{2} = 0.
\end{equation}
Lastly, from \eqref{eq:lemma_disagreement_z},
we have that: ${\lim_{k \rightarrow \infty} \norm{\tilde{\bm{z}}^{(k)}}_{2} = 0,}$
using \eqref{eq:mean_g_convergence} and \eqref{eq:v_tilde_convergence_proof}

\section{Proof of Theorem
2
}
\label{appdx:objective_convergence}
Some of ideas used in this proof arise from \cite{xu2015augmented, shorinwa2024distributed}. We provide a brief summary of the proof for completeness.
From \eqref{eq:obj_value_error_step_a}, ${\bm{f}(\mean{\bm{x}}^{(k)})}$ is bounded. As a result, we note that the coerciveness of $\bm{f}$ implies that ${\norm{\mean{\bm{x}}^{(k)}}_{2}}$ is bounded. From the convexity of $\bm{f}$:
\begin{equation}
    \begin{aligned}
        \bm{f}(\mean{\bm{x}}^{(k)}) &\leq f^{\star} + \frac{1}{N} \mathrm{trace}\left(\bm{g}(\mean{\bm{x}}^{(k)})(\mean{\bm{x}}^{(k)} - \bm{x}^{\star})
        \right) \\
        &\leq f^{\star} + \norm{\mean{\bm{g}}(\bm{x}^{(k)})}_{2} \norm{\mean{\bm{x}}^{(k)} - \bm{x}^{\star}}_{2} \\
        & \quad + \frac{L}{2} \norm{\tilde{\bm{x}}^{(k)}}_{2} \norm{\mean{\bm{x}}^{(k)} - \bm{x}^{\star}}_{2}.
    \end{aligned}
\end{equation}
Taking the limit of both sides yields:
\begin{equation}
    \label{eq:convergence_objective_value_mean_x}
    \lim_{k \rightarrow \infty} \bm{f}(\mean{\bm{x}}^{(k)}) = f^{\star},
\end{equation}
since ${\norm{\mean{\bm{x}}^{(k)} - \bm{x}^{\star}}_{2} \leq \norm{\mean{\bm{x}}^{(k)}}_{2} + \norm{\bm{x}^{\star}}_{2}}$ is bounded and $\bm{f}$ is lower-bounded by $f^{\star}$. Further, from Lipschitz-continuity of $\bm{g}$, we note that ${\norm{\bm{g}(\mean{\bm{x}}^{(k)} + \xi \tilde{\bm{x}}^{(k)})}_{2}}$ is bounded, and ${\norm{\mean{\bm{x}}^{(k)} + \xi \tilde{\bm{x}}^{(k)}}_{2}}$ is bounded, since $\norm{\mean{\bm{x}}^{(k)}}_{2}$ and $\norm{\tilde{\bm{x}}^{(k)}}_{2}$ are bounded, with ${0 \leq \xi \leq 1}.$ The mean-value theorem indicates that:
\begin{equation}
    \begin{aligned}
        \left\lvert \bm{f}(\bm{x}^{(k)}) - \bm{f}(\mean{\bm{x}}^{(k)}) \right\rvert &= \frac{1}{N} \left\lvert \mathrm{trace}\left(\bm{g}(\mean{\bm{x}}^{(k)} + \xi \tilde{\bm{x}}^{(k)}) \big(\tilde{\bm{x}}^{(k)} \big)^{\T}  \right) \right\rvert \\
        & \leq \norm{\bm{g}(\mean{\bm{x}}^{(k)} + \xi \tilde{\bm{x}}^{(k)})}_{2} \norm{\tilde{\bm{x}}^{(k)}}_{2}.
    \end{aligned}
\end{equation}
Taking the limit of both sides and applying \eqref{eq:x_tilde_convergence_proof} yields:
\begin{equation}
    \begin{aligned}
        \lim_{k \rightarrow \infty}  \left\lvert \bm{f}(\bm{x}^{(k)}) - \bm{f}(\mean{\bm{x}}^{(k)}) \right\rvert &= 0,
    \end{aligned}
\end{equation}
which results in:
\begin{equation}
    \lim_{k \rightarrow \infty} \bm{f}(\bm{x}^{(k)}) = \lim_{k \rightarrow \infty} \bm{f}(\mean{\bm{x}}^{(k)}) = f^{\star},
\end{equation}
from \eqref{eq:convergence_objective_value_mean_x}.

\begin{figure}[th]
	\centering
	\includegraphics[width=\linewidth]{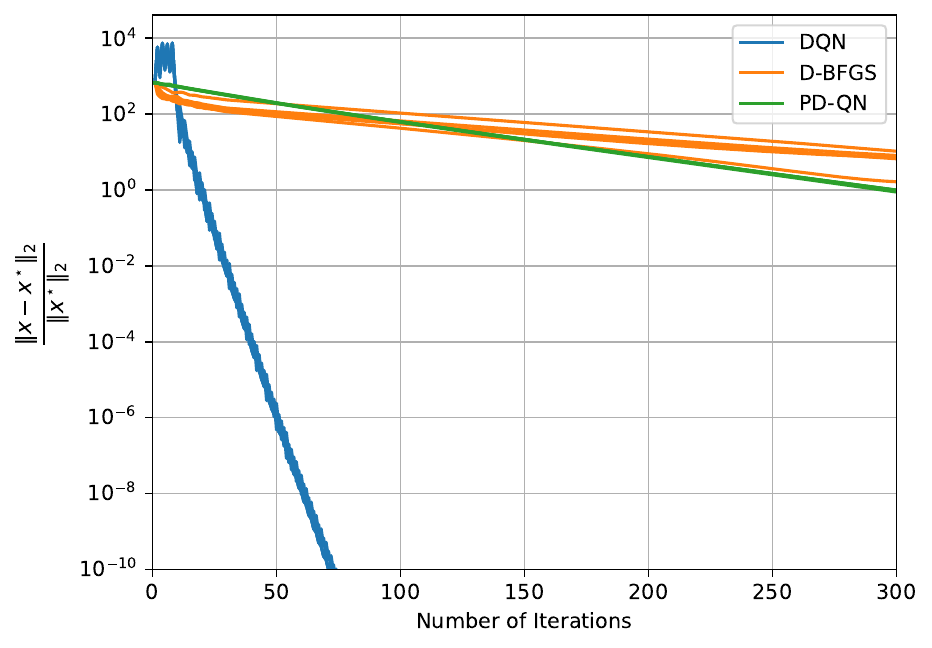}
	\caption{Convergence error of D-BFGS, PD-QN, and DQN in QP problem with ${\kappa = 0.90}$. DQN achieves the fastest convergence rate compared to the other algorithms.}
	\label{fig:d_bfgs_pdqn_stats}
\end{figure}

\begin{table}[t]
	\centering
	\caption{The mean and standard deviation of the cumulative computation time (Comp. Time), in seconds (secs.), and the cumulative size of messages exchanged per agent (Comm. Cost), in Megabytes (MB), in $20$ well-conditioned quadratic programs, across a range of communication networks.}
	\label{tab:well_conditioned_QP_d_bfgs_pdqn}
	\begin{adjustbox}{width=\linewidth}
		{\begin{tabular}{l c c c}
				\toprule
				\multirow{2}{*}{\textbf{Algorithm}} & \multicolumn{3}{c}{$\bm{\kappa} = \bm{0.70}$} \\
				\cmidrule(lr){2-4} 
				 & Comp. Time (secs.) & Comm. Cost (MB) & Success (\%) \\
				\midrule
				D-BFGS \cite{eisen2017decentralized} & $1.64\mathrm{e}^{-1}  \pm 3.17\mathrm{e}^{-2}$ & $1.54\mathrm{e}^{-1}$ & $0$
				 \\
				PD-QN \cite{eisen2019primal} & $2.48\mathrm{e}^{-1}  \pm 5.26\mathrm{e}^{-2}$ & $2.18\mathrm{e}^{-1}$ & $0$
				 \\
				DQN (ours) & $\bm{1.63\mathrm{e}^{-2}  \pm 4.46\mathrm{e}^{-3}}$ & $\bm{1.24\mathrm{e}^{-2} \pm 3.28\mathrm{e}^{-3}}$ & $\bm{100}$
				 \\
				\bottomrule
		\end{tabular}}
	\end{adjustbox}
  
	\bigskip
 
	\begin{adjustbox}{width=\linewidth}
		{\begin{tabular}{l c c c}
				\toprule
				\multirow{2}{*}{\textbf{Algorithm}} & \multicolumn{3}{c}{$\bm{\kappa} = \bm{0.90}$} \\
				\cmidrule(lr){2-4} 
				 & Comp. Time (secs.) & Comm. Cost (MB) & Success (\%) \\
				\midrule
				D-BFGS \cite{eisen2017decentralized} & $1.76\mathrm{e}^{-1}  \pm 3.73\mathrm{e}^{-2}$ & $1.80\mathrm{e}^{-1}$ & $0$
				 \\
				PD-QN \cite{eisen2019primal} & $2.68\mathrm{e}^{-1}  \pm 6.50\mathrm{e}^{-2}$ & $2.44\mathrm{e}^{-1}$ & $0$
				 \\
				DQN (ours) & $\bm{8.68\mathrm{e}^{-3}  \pm 2.11\mathrm{e}^{-3}}$ & $\bm{6.51\mathrm{e}^{-3} \pm 6.71\mathrm{e}^{-4}}$ & $\bm{100}$
				 \\
				\bottomrule
		\end{tabular}}
	\end{adjustbox}
\end{table}

\section{Comparison to D-BFGS \texorpdfstring{\cite{eisen2017decentralized}}{[10]} and PD-QN \texorpdfstring{\cite{eisen2019primal}}{[11]}}
\label{appdx:evaluation_D_BFGS_PDQN}
We evaluate the performance of DQN compared to the existing distributed quasi-Newton methods D-BFGS \cite{eisen2017decentralized} and PD-QN \cite{eisen2019primal} in well-conditioned quadratic programming (QP) problems (described in Section~\ref{sec:sim_QP}), with ${N = 5}$ and ${n = 4}$. We assess the computation time and communication cost required by each method for convergence across $20$ QP problems in two communication graphs with ${\kappa = 0.70}$ and ${\kappa = 0.90}$. We maintain the same setting for the simulations, with the convergence threshold set at $1\mathrm{e}^{-10}.$ We summarize the performance of each algorithm in Table~\ref{tab:well_conditioned_QP_d_bfgs_pdqn}. In all instances, D-BFGS and PD-QN failed to converge within $1000$ iterations. As a result, the cumulative communication cost of these methods remained constant across all problem instances on a fixed communication graph. In contrast, DQN achieves a perfect success rate, converging in every instance of the QP problem. Moreover, DQN achieves the fastest computation time with the minimum communication cost, providing at least an order of magnitude speedup compared to D-BFGS and PD-QN. 

In Figure~\ref{fig:d_bfgs_pdqn_stats}, we show the convergence error per iteration of each method on an instance of the QP problem with ${\kappa = 0.9}$. DQN converges within $75$ iterations, in contrast to D-BFGS and PD-QN, which require a significantly greater number of iterations for convergence.

\let\thesubsection = \oldsubsection
        \end{appendices}

\end{document}